\documentclass[10pt]{amsart}
\usepackage{amssymb}
\usepackage{graphicx}
\usepackage{xcolor} 
\usepackage{tensor}
\usepackage{fullpage} 
\usepackage{amsmath}
\usepackage{amsthm}
\usepackage{verbatim}
\usepackage{txfonts, bm, mathtools}
\usepackage{dsfont}
\usepackage{mathrsfs}
\usepackage{appendix}

\usepackage{enumitem}
\setlist[enumerate]{leftmargin=1.5em}
\setlist[itemize]{leftmargin=1.5em}

\usepackage{hyperref}
\usepackage{seqsplit,mathtools} 
\usepackage{caption}
\usepackage{subcaption}
\usepackage{ulem}

\usepackage{thmtools}
\usepackage{thm-restate}
\usepackage{lipsum}

\mathtoolsset{showonlyrefs}

\definecolor{green}{rgb}{0,0.5,0} 
 
\newcommand{\dd}{\textrm{d}}

\newtheorem{thm}{Theorem}[section]

\newtheorem{lem}[thm]{Lemma}

\newtheorem{rmk}[thm]{Remark}

\theoremstyle{definition}

\numberwithin{equation}{section}
\newcommand{\nrm}[1]{\Vert#1\Vert}

\newcommand{\brk}[1]{\langle#1\rangle}
\newcommand{\set}[1]{\{#1\}}
\newcommand{\tld}[1]{\widetilde{#1}}

\newcommand{\nnrm}[1]{{\vert\kern-0.25ex\vert\kern-0.25ex\vert #1 
		\vert\kern-0.25ex\vert\kern-0.25ex\vert}}

\newcommand{\supp}{{\mathrm{supp}}\,}

\newcommand{\lap}{\Delta}

\newcommand{\rd}{\partial}
\newcommand{\nb}{\nabla}

\newcommand{\0}{\emptyset}

\newcommand{\ift}{\infty}

\newcommand{\bt}{\beta}
\newcommand{\gmm}{\gamma}

\newcommand{\dlt}{\delta}

\newcommand{\eps}{\varepsilon}

\newcommand{\kpp}{\kappa}
\newcommand{\lmb}{\lambda}

\newcommand{\tht}{\theta}

\newcommand{\omg}{\omega}

\newcommand{\bfe}{{\bf e}}

\newcommand{\bfu}{{\bf u}}

\newcommand{\bfx}{{\bf x}}
\newcommand{\bfy}{{\bf y}}

\newcommand{\bbR}{\mathbb R}

\newcommand{\calB}{\mathcal B}

\newcommand{\calE}{\mathcal E}
\newcommand{\calF}{\mathcal F}
\newcommand{\calG}{\mathcal G}

\newcommand{\calI}{\mathcal I}

\newcommand{\calK}{\mathcal K}

\newcommand{\calM}{\mathcal M}

\newcommand{\calO}{\mathcal O}
\newcommand{\calP}{\mathcal P}

\newcommand{\calS}{\mathcal S}

\newcommand{\weakto}{\rightharpoonup}

\newcommand{\bfomg}{\boldsymbol{\omg}}		

\newcommand{\f}[2]{\frac{#1}{#2}}       
\newcommand{\ii}[2]{\int_{#1}^{#2}}     

\newcommand{\q}{\mbox{ }}       
\newcommand{\qd}{\quad }

\newcommand{\bfone}{\mathbf{1}}

\newcommand{\+}[1]{\brk{#1}_+}

\begin{document}

\title{Stability of oppositely-propagating pair of Hill's spherical vortices}
\author{Young-Jin Sim}
\address{Department of Mathematical Sciences, Ulsan National Institute of Science and Technology, 50 UNIST-gil, Eonyang-eup, Ulju-gun, Ulsan 44919, Republic of Korea. }
	\email{yj.sim@unist.ac.kr}
\date\today

\begin{abstract}
We establish the stability of a pair of Hill's spherical vortices moving away from each other in 3D incompressible axisymmetric Euler equations without swirl. Each vortex in the pair propagates away from its odd-symmetric counterpart, while keeping its vortex profile close to Hill's vortex. This is achieved by analyzing the evolution of the interaction energy of the pair and combining it with the compactness of energy-maximizing sequences in the variational problem concerning Hill's vortex. The key strategy is to confirm that, if the interaction energy is initially small enough, the kinetic energy of each vortex in the pair remains so close to that of a single Hill's vortex for all time that each vortex profile stays close to the energy maximizer: Hill's vortex. An estimate of the propagating speed of each vortex in the pair is also obtained by tracking the center of mass of each vortex. The estimate can be understood as optimal in the sense that the power exponent of the $\varepsilon$—the small perturbation measured in the ($L^1\cap L^2$+impulse) norm—appearing in the error bound cannot be improved.
\end{abstract}
\maketitle
\renewcommand{\thefootnote}{\fnsymbol{footnote}}
\footnotetext{\textit{2020 AMS Mathematics Subject Classification:} 35Q31}
\footnotetext{\textit{Key words: stability, Hill's vortex, anti-parallel flow, interaction energy, kinetic energy, variational problem}}
\renewcommand{\thefootnote}{\arabic{footnote}}
\section{Introduction}\q

In this paper, we are concerned with the homogeneous incompressible Euler equations in $\bbR^3$ (in vorticity form): 
\begin{equation}\label{eq: Euler eq.}
\left\{
\begin{aligned}
&\,\rd_{t}\bfomg+(\bfu\cdot\nb)\bfomg= (\bfomg\cdot\nb)\bfu\qd\mbox{in}\q [0,\ift)\times\bbR^3,\\
&\,\bfu= \nb\times
(-\lap_{\bbR^3})^{-1}\omg,\\
&\,\bfomg|_{t=0}=\bfomg_{0},
\end{aligned}
\right.
\end{equation} with a given initial data $\bfomg_0:\bbR^3\to\bbR^3$ where the \textit{velocity} $\bfu:[0,\ift)\times\bbR^3\to\bbR^3$ satisfies the incompressibility 
$\nb\cdot\bfu=0,$ and the \textit{vorticity} $\bfomg:[0,\ift)\times\bbR^3\to\bbR^3$ is given as 
$\bfomg=\nb\times\bfu.$ The expression $\bfu= \nb\times
(-\lap_{\bbR^3})^{-1}\omg$ denotes the 3D Biot-Savart law. Assuming that the velocity is axisymmetric (about $z-$axis) without swirl, the vorticity $\bfomg$ has its axisymmetric angular component $\omg^\tht$ only, i.e.,
$$\bfomg=\bfomg(r,\tht,z)=\omg^\tht(r,z)\cdot\left(-\sin\tht,\cos\tht,0\right)$$  in cylindrical coordinates $\bfx=(r,\tht,z)$. Then we can reduce \eqref{eq: Euler eq.} to the following active scalar equation
\begin{equation}\label{eq: axi no swirl Euler}
\left\{
\begin{aligned}
&\,\partial_t\xi+\bfu\cdot\nb\xi=0\qd\mbox{in}\q [0,\ift)\times\bbR^3,\\
&\,\xi|_{t=0}= \xi_0
\end{aligned}
    \right.
\end{equation}
for the \textit{relative vorticity} $\xi:[0,\ift)\times\bbR^3\to\bbR$ defined by $\xi:=\omg^\tht(r,z)/r.$ The velocity $\bfu$ is recovered from $\xi$ by the axisymmetric Biot-Savart law $\bfu=\calK[\xi]$ (see Subsection~\ref{subsec: axi-sym.no.swirl}). 

In this setting, the Hill's spherical vortex $\xi_H$ defined by 
$$\xi_H:=\bfone_{\set{\bfx\in\bbR^3:|\bfx|<1}} $$ is a traveling wave solution to \eqref{eq: axi no swirl Euler} moving along the vector $\bfe_z:=(0,0,1)$ in that, for the traveling speed $W_H=2/15$, the function $\xi=\xi(t,\bfx)$ defined by 
$$\xi(t,\bfx):=\xi_H\left(\bfx-tW_H\bfe_z\right)\qd\mbox{for}\q t\geq0,\, \bfx\in\bbR^3$$ solves \eqref{eq: axi no swirl Euler}. We refer to \cite{Hill1894} for the discovery of the Hill's vortex, written by M. J. M. Hill in 1984; see \cite{Lamb1993, Norbury1972, Norbury1973, Fraenkel1972, AF1986, Pozrikidis1986, PE2016, Wan1988, Choi2024, CJ2023, Pozrikidis1986} for various subsequent works, which will be elaborated later in Subsection~\ref{subsec: related topic}.

\subsection{Main statement}\q

Our main result consists of Theorem~\ref{thm: main result with no shift estimate} and Theorem~\ref{thm: shift estimate}. Theorem~\ref{thm: main result with no shift estimate} below is the stability of a distant, odd-symmetric pair of Hill's spherical vortices $\xi_H$. In the statement, we use the notations
$$L^1_w(\bbR^3):=\left\{
f=(x_1^2+x_2^2)^{-1}g:\, g\in L^1(\bbR^3)
\right\}\qd\mbox{with}\qd
\nrm{f}_{L^1_w(\bbR^3)}:=
\int_{\bbR^3}(x_1^2+x_2^2)|f(\bfx)|\,\dd\bfx.
$$
\begin{thm}(Stability of odd-symmetric pair of Hill's vortices)\label{thm: main result with no shift estimate}
For each $\eps>0$, there exist $\dlt=\dlt(\eps)>0$ and $d_0=d_0(\eps)>1$ satisfying the following: \\
    
For any $d\geq d_0$ and for any axisymmetric initial data $\xi_0\in \left(L^1\cap L^\ift\cap L^1_w\right)(\bbR^3)$ with $r\xi_0\in L^\ift(\bbR^3)$ such that
$$\xi_0(r,z)=-\xi_0(r,-z)\geq0\qd\mbox{for}\q r>0,\,z\geq0$$ and
$$\left\|\xi_0-\left[\xi_H(\cdot-d\bfe_z)-\xi_H(\cdot+d\bfe_z)\right]\right\|_{\left(L^1\cap L^2\cap L^1_w\right)(\bbR^3)}<\dlt,$$  there exists a shift function $\tau:[0,\ift)\to[1,\ift)$ such that $\tau(0)=d$ and the axisymmetric weak solutions $\xi(t)$ of \eqref{eq: axi no swirl Euler} satisfies 
    $$\left\|\xi(t)-\left[\xi_H\left(\cdot-\tau(t)\bfe_z\right)-\xi_H\left(\cdot+\tau(t)\bfe_z\right)\right]\right\|_{\left(L^1\cap L^2\cap L^1_w\right)(\bbR^3)}<\eps\qd\mbox{for each}\q t\geq0.$$ Here, the notation $\nrm{\cdot}_{L^1\cap L^2\cap L^1_w}$ means $\nrm{\cdot}_{L^1}+\nrm{\cdot}_{L^2}+\nrm{\cdot}_{L^1_w}$.
\end{thm}

Theorem~\ref{thm: main result with no shift estimate} above will be proved later at the end of Section~\ref{sec: stability}. See Figure~\ref{fig} for the simple illustration of Theorem~\ref{thm: main result with no shift estimate} compared to the case of a single Hill's vortex.

Theorem~\ref{thm: shift estimate} below contains the shift function estimate. That is, with some additional assumptions on $\xi_0$, we establish that any shift function $\tau:[0,\ift)\to[1,\ift)$ in Theorem~\ref{thm: main result with no shift estimate} propagates approximately with the speed $W_H=2/15$, which is the traveling speed of Hill's spherical vortex $\xi_H$, i.e., 
$$\tau(t)-\tau(0)\simeq W_Ht.$$

\begin{figure}
    \centering
    \includegraphics[width=0.95\linewidth]{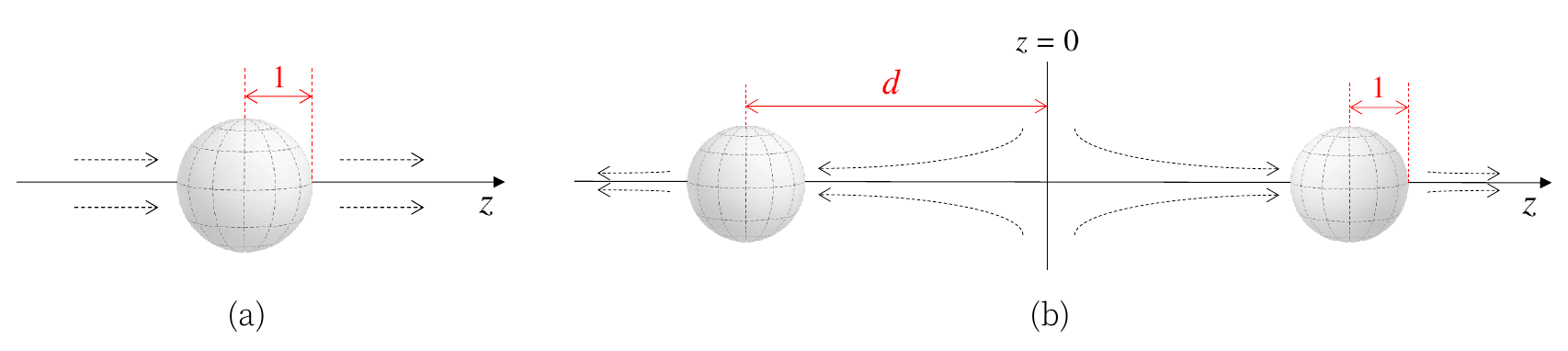}
    \caption{Theorem~\ref{thm: main result with no shift estimate} establishes the stability of a distant, odd-symmetric pair of Hill's vortices as demonstrated in (b). The dashed arrows in the figures demonstrate that, unlike the case of a single Hill's vortex in (a), two Hill's vortices consisting the initial vorticity $\xi_0$ in (b) move away from each other while experiencing an all-time decrease of the impulse $\int r^2|\xi(t)|$ due to odd-odd symmetry of the vorticity; see Lemma~\ref{lem: monoticity of impulse} for details.}
    \label{fig}
\end{figure}

\begin{thm}(Shift function estimate)\label{thm: shift estimate} 
For any $\calM>0$, there exists a constant $\eps_0=\eps_0(\calM)>0$ such that the following holds:\\

For each $\eps\in(0,\eps_0)$, let $d_0=d_0(\eps)>1$ be the corresponding constant given in Theorem~\ref{thm: main result with no shift estimate} (If necessary, we redefine
$d_0(\eps)>\eps^{-1}$). Let $d\geq d_0$ be any constant and 
$\xi_0$ be any axisymmetric initial data satisfying the assumptions in Theorem~\ref{thm: main result with no shift estimate} such that $\nrm{\xi_0}_{L^\ift(\bbR^3)}\leq \calM$ and the support of $\xi_0$ is contained in  $\set{\bfx\in\bbR^3:\, d-2\leq|x_3|\leq d+2}.$ Then, any shift function $\tau:[0,\ift)\to[1,\ift)$ in Theorem~\ref{thm: main result with no shift estimate} satisfies the following estimate for some constant $C=C(\calM)>0$:
\begin{equation}\label{(revisited) eq: tau(t) vs Wt}
    \left|\tau(t)-\tau(0)-W_Ht\right|<  
        C\eps(t+1)
        \qd\mbox{for each }t\geq0
\end{equation} where $\tau(0)=d$ and $W_H=2/15$ is the traveling speed of the (single) Hill's spherical vortex $\xi_H$.
   \end{thm}
Theorem~\ref{thm: shift estimate} above will be proved later in Section~\ref{sec: shift}.

\begin{rmk}
    The shift estimate \eqref{(revisited) eq: tau(t) vs Wt} in Theorem~\ref{thm: shift estimate} above is sharp, in the sense of the power exponent of $\eps$ (which is $1$) in the right-hand side. Indeed, when considering the odd-pair of $\lmb\xi_H$ ($\lmb$-strength Hills vortex) as a perturbed data in Theorem~\ref{thm: main result with no shift estimate}, its $L^1\cap L^2\cap L^1_w$ error ($\sim|\lmb-1|$) in Theorem~\ref{thm: main result with no shift estimate} coincides with the shift error ($\sim |\lmb-1|(t+1)$), knowing that the traveling speed of  $\lmb\xi_H$ is given as $\lmb W_H$. We note that the method used to obtaining the estimate \eqref{(revisited) eq: tau(t) vs Wt} can be used to improve the shift function estimate given in \cite[Proposition 3.2]{CJ2023}, where the authors considered a single Hill's vortex, in the optimal way of raising the power exponent of $\eps$ from $1/2$ to $1$. 
\end{rmk}

\begin{rmk}
Thanks to the shift function estimate \eqref{(revisited) eq: tau(t) vs Wt} together with the stability given in Theorem~\ref{thm: main result with no shift estimate}, one can prove the linear-in-time filamentation near the odd-symmetric pair of the Hill's vortices moving away from each other nearly with speed $W_H=2/15$. Indeed, we refer to \cite{CJ2023} for the result of filamentation near a \textit{single} Hill's vortex. It was obtained from the stability of a single Hill's vortex (given in \cite{Choi2024}) and an additional shift function estimate. For the similar filamentation result near the single Chaplygin-Lamb's dipole in $\bbR^2$, we refer to \cite{CJ2022}, and see also \cite{CSW2025} for the filamentation result near a collection of the Sadovskii vortex patches. 
\end{rmk}

\subsection{3D axisymmetric Euler equations without swirl}\label{subsec: axi-sym.no.swirl}\q

For the velocity $\bfu:\bbR^3\to\bbR^3$ in the cylindrical form
$$\bfu=\bfu(r,\tht,z)=u^r(r,\tht,z)\bfe_r(\tht)+u^\tht(r,\tht,z)\bfe_\tht(\tht)+u^z(r,\tht,z)\bfe_z$$ with the cylindrical coordinate $(r,\tht,z)$ where the unit vectors are defined by $$\bfe_r(\tht):=(\cos\tht,\sin\tht,0),\q \bfe_\tht(\tht):=(-\sin\tht,\cos\tht,0),\q
\bfe_z=(0,0,1),$$ we say the velocity is \textit{axisymmetric (about $z$-axis)} when each component $u^r,u^\tht,u^z$ does not depend on $\tht$. The \textit{no-swirl} condition stands for $u^\tht\equiv0$. In this setting, the velocity $\bfu$ has the form 
$$\bfu=\bfu(r,\tht,z)=u^r(r,z)\bfe_r(\tht)+u^\tht(r,z)\bfe_\tht(\tht)+u^z(r,z)\bfe_z,$$ and the vorticity $\bfomg:\bbR^3\to\bbR^3$ has the form 
$$\bfomg=\nb\times\bfu=\omg^\tht(r,z)\bfe_\tht(\tht)\qd\mbox{where}\q \omg^\tht:=\rd_z u^r-\rd_r u^z.$$ By defining the relative vorticity $\xi$ by $$\xi(r,z):=\f{1}{r}\omg^\tht(r,z),$$ 
the velocity is recovered by the following axisymmetric Biot-Savart law: 
$$\bfu=\calK[\xi]:=\nb\times\left(\f{1}{r}\calG[\xi]\bfe_\tht\right)$$
 where the function $\calG[\xi]$ is defined as  $$\calG[\xi]=\calG[\xi](r,z):=\int_\Pi G(r,z,r',z')\xi(r',z')r'dr'dz'$$for the domain $\Pi:=\set{(r,z)\in\bbR^2:\, r>0)}$ with the Green's function given as 
\begin{equation}\label{eq: Green function formula}
G(r,z,r',z')=\f{rr'}{2\pi}\ii{0}{\pi}\f{\cos\vartheta}{\sqrt{r^2+r'^2-2rr'\cos\vartheta+(z-z')^2}}\dd\vartheta\q>0.
\end{equation} In specific, we have 
\begin{equation}\label{eq: axi. sym. biot-savart law}
    \bfu=\calK[\xi]=\f{-\rd_z\calG[\xi]}{r}\bfe_r(\tht)+\f{\rd_r\calG[\xi]}{r}\bfe_z.
\end{equation}
Then the 3D Euler equations \eqref{eq: Euler eq.} are reduced to \eqref{eq: axi no swirl Euler}: $$\rd_t\xi+\bfu\cdot\nb\xi=0\qd\mbox{in}\q [0,\ift)\times\bbR^3\qd\mbox{where }\bfu=\calK[\xi]$$ with the initial data $\xi_0$. It is well-known that, given axisymmetric initial data $\xi_0\in \left(L^1\cap L^\ift\right)(\bbR^3)$ satisfying $\xi_0\in L^1_w(\bbR^3)$ and $r\xi_0\in L^\ift(\bbR^3)$, there exists a unique axisymmetric weak solution $\xi\in L^\ift\left(0,\ift;\left(L^1\cap L^\ift\right)(\bbR^3)\right)$ of \eqref{eq: axi no swirl Euler} (see \cite[Remark 3.5]{Choi2024}). The solution has several quantities that are conserved in time:\\

\noindent\textit{(i)} $L^p$ norms: $$\nrm{\xi(t)}_{L^p(\bbR^3)}\qd\mbox{for each }p\in[1,\ift]$$
    \noindent\textit{(ii)} Kinetic energy:
    $$E[\xi(t)]:=\f{1}{2}\int_{\bbR^3}\xi(t)\calG[\xi(t)]\,\dd\bfx$$
    \noindent\textit{(iii)} Distribution: 
    $$\int_{\set{\bfx\in\bbR^3: a<\xi(t)<b}}\xi(t)\,\dd\bfx\qd\mbox{for each}\q-\ift<a<b<\ift.$$
Moreover, if $\xi_0$ has the odd symmetry with respect to the plane $\set{x_3=0}$:
    $$\xi_0(r,z)=-\xi_0(r,-z)\geq0\qd\mbox{for}\q z\geq0,$$ its solution $\xi(t)$ also has the odd symmetry: for each $t\geq0$, we have
    $$\xi(t,(r,z))=-\xi(t,(r,-z))\geq0\qd\mbox{for}\q r>0,\, z\geq0.$$ 

\subsection{Key ideas}\q

\subsubsection{Interaction between two vortices} \q

For the governing equation \eqref{eq: axi no swirl Euler}, the kinetic energy $E[\xi]$ is conserved in time $t\geq0$, and has a bilinear form in $\xi$:
$$E[\xi]=\calE[\xi,\xi]:=\pi\iint_{\Pi\times\Pi} G(r,z,r',z')\xi(t,(r,z))\xi(t,(r',z'))rr'\, \dd r\dd z\dd r'\dd z',
$$ for the Green's function $G$ with the domain $\Pi:=\set{(r,z)\in\bbR^2:\, r>0)}$. Using this conservation, \cite{Choi2024} obtained the stability of single Hill's spherical vortex $\xi_H$:
\begin{thm}\cite[Theorem 1.2]{Choi2024}
    For $\eps>0$, there exists $\dlt>0$ such that for any non-negative axisymmetric function $\xi_0$ satisfying 
    $$\xi_0,\,r\xi_0\in L^\ift(\bbR^3)$$ and 
    $$\nrm{\xi_0-\xi_H}_{L^1\cap L^2(\bbR^3)}+\nrm{r^2(\xi_0-\xi_H)}_{L^1(\bbR^3)}\leq\dlt,$$ the corresponding solution $\xi(t)$ of \eqref{eq: axi no swirl Euler} satisfies 
    $$\inf_{\tau\in\bbR}\left\{
    \nrm{\xi(\cdot+\tau\bfe_z)-\xi_H}_{L^1\cap L^2(\bbR^3)}+ 
    \nrm{r^2(\xi(\cdot+\tau\bfe_z)-\xi_H)}_{L^1(\bbR^3)}\right\}\leq\eps,\q\forall t\geq0.
    $$
\end{thm}
The main idea was that, for any solution $\xi(t)$, if its kinetic energy $E[\xi(t)]\simeq E[\xi_H]$ is conserved in time, then the global-in-time stability is obtained: $$\xi\simeq\xi_H\q\mbox{up to translation},$$ based on some energy-maximization problem concerning Hill's vortex $\xi_H$ (See Subsection~\ref{subsec: var. prob. Hills vortex}). In this paper, on the other hand, we assume that the solution has odd symmetry
$$\xi(t,(r,z))=-\xi(t,(r,-z))\geq0\qd\mbox{for }r>0,z\geq0,$$ we decompose the solution $\xi=\xi(t)$ as 
$$\xi=\xi^+-\xi^-\qd\mbox{for}\qd
\xi^+=\xi\big|_{z>0}\q\mbox{and}\q
\xi^-=-\xi\big|_{z<0},$$
to observe that
\begin{equation}\label{eq: energy relation}
    E[\xi(0)]=E[\xi(t)]=
2\calE\left[\xi^+(t),\xi^+(t)\right]
-2\calE\left[\xi^+(t),\xi^-(t)\right]
=2\underbrace{E[\xi^+(t)]}_{>0}
-2\underbrace{E_{inter}(t)}_{>0}
\end{equation}
for the interaction energy $$E_{inter}(t)=\calE[\xi^+(t),\xi^-(t)]>0$$
(Its positivity stems from the odd symmetry of $\xi$). Here we can see that, despite the conservation of total kinetic energy 
$E[\xi]=E\left[\xi^+-\xi^-\right]$, the kinetic energy of the non-negative part $E[\xi^+]$ and the interaction energy $E_{inter}$ are not conserved in time.
The fact that $E[\xi^+]$ is not conserved is a main obstacle in proving the stability of two oppositely signed Hill's spherical vortices in the same way as \cite{Choi2024}. We overcome this obstacle based on the strategy of \textit{keeping $E_{inter}(t)$ small enough for all time}
and so maintaining the kinetic energy $E[\xi^+(t)]$ close enough to the initial energy $E[\xi^+(0)]$, though it is not conserved. We obtain $E_{inter}(t)\ll1$ by following steps:\\

[Step 1] In the relation 
$$0< E_{inter}(t)= E[\xi^+(t)]-E[\xi^+(0)]+E_{inter}(0),$$ we choose initial data $\xi^+(0)$ close enough to $\xi_H$ (up to translation) and obtain 
$$
E[\xi^+(0)]\simeq E[\xi_H].
$$ 

[Step 2] We prove that 
$$\nrm{r^2\xi^+(t)}_{L^1}<\nrm{r^2\xi^+(0)}_{L^1}\qd\mbox{for }t\geq0$$ (see Lemma~\ref{lem: monoticity of impulse}), which would turn out to imply 
$$E[\xi^+(t)]< E[\xi_H]+\eps\simeq  E[\xi^+(0)]+\eps$$ with some small error $\eps\ll1$. This is achieved by showing that the decrease of the impulse $\nrm{r^2\xi^+(t)}_{L^1}$ prevents the solution $\xi^+(t)$ from possessing the kinetic energy much larger than the Hill's vortex $\xi_H$: the energy maximizer under the fixed impulse.\\

[Step 3] Roughly speaking, we have obtained 
$$0<E_{inter}(t)<E_{inter}(0)+\eps.$$
We also get $E_{inter}(0)=\calE\left[\xi^+(0),\xi^-(0)\right]\ll1$ by assuming that the supports of $\xi^+(0)$ and $\xi^-(0)$ are far away from one another. Hence, we get $$E_{inter}(t)\ll1.$$ These sorts of ideas can be found in \cite{CJY2024}, which dealt with the odd-pair of the Chaplygin-Lamb dipole. 

\subsubsection{Shift function estimate}\q

As a function of time $t\geq0$, the shift function $\tau:[0,\ift)\to[1,\ift)$ in Theorem~\ref{thm: main result with no shift estimate} refers to the $z$-translation of the Hill's spherical vortex $\xi_H$ to which \textit{the positive part of the solution $\xi(t)$} is close enough, i.e.,
$$\xi^+(t)=\xi(t)\big|_{z>0}\simeq \xi_H(\cdot-\tau(t)\bfe_z).$$ Since $\xi_H$ travels in time with the constant velocity $W_H\bfe_z$ with $W_H=2/15$, we expect that \textit{the shift function $\tau(t)$ is close enough to $W_H t$ in a finite time interval} (Theorem~\ref{thm: shift estimate}), i.e.,
$$\tau(t)-\tau(0)\simeq W_H t.$$
Recalling that $\xi_H$ has even symmetry: 
$$\xi_H(r,z)=\xi_H(r,-z)\qd\mbox{for}\q (r,z)\in\Pi$$ which implies $\int_{\bbR^3}x_3\xi_H\dd\bfx=0,$ we note that the shift function $\tau(t)$ is obtained by the relation 
\begin{equation}\label{eq0324_1}
    \tau(t)= \nrm{\xi_H}_{L^1(\bbR^3)}^{-1}
\int_{\bbR^3}x_3\xi_H(\cdot-\tau(t)\bfe_z)\,\dd\bfx\,\simeq\,
\nrm{\xi^+(t)}_{L^1(\bbR^3)}^{-1}\int_{\bbR^3}x_3\xi^+(t)\,\dd\bfx
\end{equation}
implying that the \textit{center of mass on $z$-axis of the solution $\xi^+(t)$ is close to the translation $\tau(t)$}. Similarly, the traveling speed $W_H$ of the Hill's vortex $\xi_H$ is obtained by the relation 
\begin{equation}\label{eq0324_2}
    W_H t= \nrm{\xi_H}_{L^1(\bbR^3)}^{-1}
\int_{\bbR^3}x_3\xi_H(\cdot-tW_H\bfe_z)\,\dd\bfx.
\end{equation} Using the observations \eqref{eq0324_1} and \eqref{eq0324_2}, we may obtain $\tau(t)-\tau(0)\simeq W_H t$ by proving that 
\begin{equation}\label{eq0324_3}
    \nrm{\xi^+(t)}_{L^1(\bbR^3)}^{-1}\int_{\bbR^3}x_3\xi^+(t)\,\dd\bfx-\q\tau(0)\q\simeq\q \nrm{\xi_H}_{L^1(\bbR^3)}^{-1}
\int_{\bbR^3}x_3\xi_H(\cdot-tW_H\bfe_z)\,\dd\bfx.
\end{equation}
Here, one may conceive a complication in that $\xi^+(t)$ is the positive part of the actual solution $\xi(t)$ of \eqref{eq: axi no swirl Euler}, while $\xi_H(\cdot-tW_H\bfe_z)$ denotes the genuine (traveling) solution of \eqref{eq: axi no swirl Euler}. However, we can expect the following: \\

\textit{As long as the positive part $\xi^+(t)$ and negative part $\xi^-(t)$ are at a long distance, i.e.,
$\tau(t)\gg1$, the dynamics of $\xi^+(t)$ and $\xi_H(\cdot-tW_H\bfe_z)$ are somehow similar in the sense of \eqref{eq0324_3}.} \\

\noindent Such an expectation is justified if one considers the time-derivative on both sides of \eqref{eq0324_3} and the integration by parts in both integrals. Essentially, we come up with a bootstrap argument with the following steps:\\

[Step 1] At the initial time $t=0$, we assume that  $\xi^+(0)$ and $\xi^-(0)$ are distant to each other, i.e.,
$$\tau(0)\gg1,$$
so that \eqref{eq0324_3} holds in a certain time interval $[0,T]$.\\

[Step 2] In the interval $[0,T]$, the relation \eqref{eq0324_3} gives $$\tau(T)>\tau(0)\gg1,$$ so that we go back to [Step 1] and repeat the same process. \\

[Step 3] We obtain that $\tau(t)\gg1$ for all $t\geq0$, and so we obtain \eqref{eq0324_3} for all $t\geq0$. It gives 
$$\tau(t)-\tau(0)\simeq W_H t.$$
We refer to \cite{CSW2025} for a similar idea to obtain the shift estimate in two-dimensional flow, concerning the stability of the Sadovskii vortex patches constructed in \cite{CJS2024} by a variational approach. See also \cite{CJ2023} for another way to estimate the translation of perturbed data near a single Hill's vortex. 

\subsection{Related topics}\label{subsec: related topic}\q 

\subsubsection{Concerning Hill's spherical vortex}\q 

Hill's spherical vortex, initially introduced by M. Hill in  \cite{Hill1894}, is an exact traveling wave solution to the equations of inviscid, incompressible, axisymmetric fluid flow. In this model, a spherical region $\set{\bfx\in\bbR^3:\, |\bfx|<1}$ contains a nonzero vorticity, while the surrounding fluid remains irrotational.  Further analytical treatment was provided in Hydrodynamics \cite{Lamb1993}, formalizing the spherical vortex as a steady, axisymmetric solution to Euler’s equations. It is later extended to a one-parameter family of vortex rings, incorporating Hill’s vortex as a limiting case, together with some numerical descriptions (\cite{Norbury1972, Norbury1973}). Shapes and properties of steady vortex rings with small cross-sections were explored in \cite{Fraenkel1972}, emphasizing their connection to Hill's vortex. 

On the other hand, several studies have focused on the properties of Hill's vortex, including the uniqueness and stability/instability. In \cite{AF1986}, it was rigorously proved that Hill’s vortex is the unique (weak) solution to Hill's problem of finding steady vortex rings under certain choices of vorticity function $\left(f(\cdot)=\bfone_{\set{\cdot>0}}\right)$ and the flux constant ($\gmm=0$). A nonlinear evolution of several axisymmetric perturbations to Hill’s spherical vortex was observed in \cite{Pozrikidis1986} using numerical simulations, indicating its nonlinear instability. The linear stability of Hill's vortex to axisymmetric perturbations is given in \cite{PE2016}, based on a spectral approach. In \cite{Wan1988}, it was shown that the Hill's vortex is a nondegenerate local maximum of the kinetic energy among vortex rings subject to a fixed impulse. Recently, the Lyapunov stability was obtained in \cite{Choi2024}, showing that Hill’s vortex is the unique maximizer of the kinetic energy under fixed impulse and circulation constraints, \textit{which is one of the main ingredients of the current paper.} Later, the linear-in-time filamentation near Hill's vortex was shown in \cite{CJ2023}, which confirms the numerical result of nonlinear instability given in  \cite{Pozrikidis1986}, based on the stability of Hill's vortex obtained by the first author, combined with additional shift function estimates and some bootstrap arguments. 

\subsubsection{Vortex quadrupoles in two-dimensional flow}\q

Due to the similarity between the settings of 2D flow and of the 3D axisymmetric flow without swirl, we give some previous results that can be seen as 2D counterparts of the 3D axisymmetric result given in Theorem~\ref{thm: main result with no shift estimate}, i.e., some results regarding 2D vorticity solutions satisfying 
\begin{equation}\label{eq0723_1}
    \omg(x_1,x_2)=-\omg(x_1,-x_2)=-\omg(-x_1,x_2)=\omg(-x_1,x_2)\geq0\qd\mbox{for}\q x_1,x_2>0.
\end{equation}
We first refer to \cite{CJY2024} for some two-dimensional results. They proved the stability of several vortex-quadrupoles having odd-odd symmetry and the nonnegativity in the first quadrant. Two kinds of vorticity (in the first quadrant) were addressed: \textit{sufficiently concentrated vorticity around a point} and \textit{the Chaplygin-Lamb dipole}. The stability of the latter was obtained mainly by the monotonicity of the first moment given in \cite{ISG1999}, some careful analysis of the interaction energy between the vortices, and the stability result of a single Chaplygin-Lamb dipole given in \cite{AC2022}. Especially, much of the essential structure and ideas of the proof of Theorem~\ref{thm: main result with no shift estimate} have been taken from the stability of a pair of Chaplygin-Lamb dipoles 
given in \cite{CJY2024}. We refer to \cite{DPMP2023} for the existence result of similar quadrupoles, which has the form of the desingularization of 4 point-vortices, together with some specific analysis of the profile of their vortex quadrupoles. See also \cite{Yang2021} for the result regarding the behavior of 4 point-vortices with odd-odd symmetry.

\subsubsection{Multi-vortex problems}\q

As the long-time dynamics of the Euler flow is a highly challenging problem in general, the multi-vortex problems are expected to play a role in understanding the evolution of the Euler flow as $t\to\ift$. It is to ask whether several coherent vortex structures--each of which would move coherently (translate, rotate, etc) if isolated--can be arranged so that they keep their shapes while their mutual interaction weakens over time, typically because they separate.

The key strategy is to construct a configuration of well‑separated vortex objects and prove that the residual interaction is too small to destroy their rigidity/coherency. In vortex dynamics, canonical building blocks might include translating dipoles (e.g., Chaplygin-Lamb dipole, Sadovskii vortex patches), rotating V-states (e.g., Rankine vortex, Kirchhoff ellipse), axisymmetric vortex rings including Hill's vortex, and desingularized point vortex (or point-vortex ring). The question is whether copies of such structures can coexist, either moving apart or forming choreographed motions, without significant deformation. Several recent works have used the Chaplygin-Lamb dipole as a building block to construct and solve multi-vortex problems; see \cite{CJY2024} for the case of the dipole pair moving away from each other, and see \cite{AJY2025} for the stability of finite sums of the dipoles moving together in one direction, in which the faster dipoles are positioned to the right of the slower ones.

Several equilibria or explicit dynamics of point vortex system  (e.g., various pairs, Thomson polygons, Von K\'arm\'an vortex streets, self-similar expanding triples, leapfrogging motions) provide a sketch of multi-vortex configurations in many cases (see \cite{Helmholtz1858, Kirchhoff1883} for the initial works of the point vortex system, followed by a vast amount of subsequent works; e.g., see \cite{HKILB1988, Newton2001, Aref2007, Lin1941, Lin1941_2}). Starting from well-known configurations such as odd (or even) dipole (see  \cite{Tur1983, HM2017, HH2021, CLZ2021, GH2023, HH2021} and references therein), the system has several equilibria consisting of multiple point vortices. 

For example, a point-vortex quadruple with odd-odd symmetry can be viewed as a multi-vortex configuration comprised of an odd pair of traveling dipoles. The quadruple has an explicit trajectory (\cite{Lamb1993, Yang2021, CJY2024}), around which any of the desingularized quadrupole stays (see \cite[Theorem 1.2, Corollary 1.3]{CJY2024} and \cite{DPMP2023}). As a 3D analogy, head-on collision of anti-parallel vortex rings is a typical example and has gained enormous interest due to its significance in probing maximal vortex-stretching and possible blow-up in 3D incompressible Euler flow; see \cite{Oshima1978, SLZF1988, SL1992, CWCCC1995, GWRW2016, CLL2018, CJ2021} and references therein. Thomson’s polygon, $N$ point vortices located at the vertex of a regular polygon with $N$ sides, can be seen as a multi-vortex version of a co-rotating pair (e.g., see \cite{KY2022, Dhanak1992, Havelock1931, Kurakin2010, Kurakin2014, KO2022, CS1999, BHW2012, CLP2011} and references therein), and it was desingularized in \cite{Garcia2021}(see also \cite{HW2022}. For gSQG equations, we refer to \cite{GGS2023}) to obtain rotating $N$-patch choreographies. On the other hand, there also exist self-similarly expanding/contracting configurations composed of 3 (or 4,5) point vortices, known as the vortex collapse (see \cite{NS1979, LLZ2000, Aref1979, Aref2007}). For the case of expanding triples of points, see two recent papers \cite{DPMP2024, Zbarsky2021}. Moreover, for the 2D/3D leapfrogging motions, we refer to \cite{HHM2023, BG2019, GSPBB2021, Aiki2019, Acheson2000, TA2013, DPM2024, DHLM2025, BCM2025, CFQW2025} and references therein. These works show that many multi‑patch rigid motions survive beyond the idealized point‑vortex level. 

\subsection{Open questions}\q

We establish some open questions related to the current paper.

\subsubsection{Asymptotic behavior.}
Compared to the case of the single Hill's vortex, the variational approach for the odd-symmetric pair of Hill's vortices seems to demonstrate the possibility of a time-asymptotic behavior of perturbed initial data under some proper constraints. Here, we give some observations.\\ 

\textit{Observation 1: The energy of a single vortex is convergent as }
$$E[\xi^+(t)]\to \f{1}{2}E[\xi(0)]\qd\mbox{as}\q t\to\ift.$$ 
Theorem~\ref{thm: shift estimate} implies that, for any perturbed initial data $\xi_0$, the interaction energy given by 
$$E_{inter}(t)=\calE\left[\xi^+(t),\xi^-(t)\right]=\pi\iint_{\Pi\times\Pi}
G(r,z,r',z')\xi^+(t,(r,z))\xi^-(t,(r',z'))\q \dd r \dd z \dd r' \dd z'$$ vanishes as $t\to\ift$. Then, from the energy relation \eqref{eq: energy relation}:
$$E[\xi(t)]=2E[\xi^+(t)]-2E_{inter}(t)$$ where the total energy is conserved as $E[\xi(t)]=E[\xi(0)]$ for any $t\geq0$, it holds that the energy of  the single vortex, $E[\xi^+(t)]$, satisfies 
$$E[\xi^+(t)]\to \f{1}{2}E[\xi(0)]\qd\mbox{as}\q t\to\ift.$$ 

\textit{Observation 2: The impulse of a single vortex is convergent in that, for some $\mu>0$, we have 
$$\nrm{r^2\xi^+(t)}_{L^1}\to\mu\qd\mbox{as}\q t\to\ift.$$ } It follows from the facts that $\nrm{r^2\xi^+(t)}_{L^1}$ decreases in $t\geq0$ (by Lemma~\ref{lem: monoticity of impulse}) and that $\nrm{r^2\xi^+(t)}_{L^1}$ is bounded from below (due to $\left|\nrm{r^2\xi^+(t)}_{L^1}-\nrm{r^2\xi_H}_{L^1}\right|\ll1$).\\

The above observations imply that both energy $E[\xi^+(t)]$ and the impulse $\nrm{r^2\xi^+(t)}_{L^1}$ are convergent. Suppose we can find a way to characterize the impulse limit $\mu>0$ and find $\xi(0)$ such that the energy limit $\f{1}{2}E[\xi(0)]$ is the maximal energy under the impulse constraint $\mu$.  Then we may use the concentrated-compactness argument given in Theorem~\ref{thm: compactness of maximizer set} to prove that the solution $\xi^+(t)$ asymptotically approaches a maximizer set. After certain scaling, the Hill's vortex is a unique maximizer (Theorem~\ref{thm: uniqueness of maximizer}), implying that the solution $\xi^+(t)$ tends to the (scaled) Hill's vortex, up to translations. 

\subsubsection{Non-symmetric perturbed solutions}

Theorem~\ref{thm: main result with no shift estimate} concerns with initial data $\xi_0$ satisfying the symmetry 
$$\xi_0(r,z)=-\xi_0(r,-z)\geq0\qd\mbox{for}\q r,z\geq0.$$ Such odd-symmetry easily implies that the impulse $\nrm{r^2\xi(t)}_{L^1}$ decreases for all $t\geq0$, which would guarantees that the interaction energy $E_{inter}(t)=\calE[\xi^+(t),\xi^-(t)]$ tends to $0$ as $t\geq\ift$, so that each of $\xi^+(t),\xi^-(t)$ stays close to the Hill's vortex up to translations.

We now suppose that $\xi_0$ only has the sign condition: 
$$\xi_0(r,z),\,-\xi_0(r,-z)\geq0\qd\mbox{for}\q r,z\geq0,$$ while $\xi_0$ is close to the odd-symmetric pair of Hill's vortex. In Lemma~\ref{lem: monoticity of impulse}, the decrease of impulse $\nrm{r^2\xi(\cdot)}_{L^1}$ relies heavily on such a sign condition, not on the symmetry itself, (at least) up to a finite time. We believe that one can prove the global-in-time stability after proper bootstrap arguments. 

\subsubsection{Other types of multi-vortex}
Even though the current paper concerns the pair of Hill's vortex, it can be generalized to other types of vortices that are characterized as energy maximizers under certain constraints; e.g., see \cite[Remark 2.9]{CJY2024} for a similar discussion. One can also try to increase the number of vortices to more than two, provided that the interaction decreases under certain configurations. As far as the interaction energy vanishes, the vortices will nearly retain their initial shapes, implying their stability.

\subsection{Organization of paper}\q

Section~\ref{sec: preliminary} contains some background and preliminaries. The stability (Theorem~\ref{thm: main result with no shift estimate}) is proved in Section~\ref{sec: stability}. Based on the stability result, we obtain the shift function estimate (Theorem~\ref{thm: shift estimate}) in Section~\ref{sec: shift}.

\section{Preliminaries}\label{sec: preliminary}

\subsection{Notations}\q
\begin{itemize}
    \item We define the half-space $\bbR^3_+$ as 
$$\bbR^3_+:=\left\{\bfx\in\bbR^3:\, x_3>0\right\}.$$
\item We define $\+{\cdot}:\bbR\to\bbR$  as $$\+{s}:=\max\set{s,0}.$$
\item We define $$\Pi:=\left\{(r,z)\in\bbR^2:\, r>0\right\}
\qd\mbox{and}\qd Q:=\left\{(r,z)\in\Pi:\, z>0\right\}.$$
\item We denote each norm of $L^p(\bbR^3)$ for $p\in[1,\ift)$ as 
$$\nrm{f}_p:= \nrm{f}_{L^p(\bbR^3)}.$$
\item  We define the weighted space as 
$$L^1_w(D):=\left\{ f=(x_1^2+x_2^2)^{-1}g:\, g\in L^1(D)\right\}\qd\mbox{for given domain}\q D\subset\bbR^3$$ with the norm
$$\nrm{f}_{L^1_w(D)}:=\int_D(x_1^2+x_2^2)|f(\bfx)|\,\dd\bfx.$$ For convenience, for the domain $D=\bbR^3$, we denote the norm of $L^1_w(\bbR^3)$ as 
$$\nrm{r^2f}_1:=\nrm{f}_{L^1_w(\bbR^3)}$$ which will be called as \textit{impulse}. 
\end{itemize}

\subsection{Anti-parallel flow}\q

In subsection~\ref{subsec: axi-sym.no.swirl}, the 3D axisymmetric Euler equations with no swirl have the reduced form \eqref{eq: axi no swirl Euler}: $$\rd_t\xi+\bfu\cdot\nb\xi=0\qd\mbox{in}\q [0,\ift)\times\bbR^3\qd\mbox{where }\bfu=\calK[\xi]$$ with a given initial data $\xi_0$. Specifically, we deal with the solution that has odd symmetry: 
    $$\xi(t,(r,z))=-\xi(t,(r,-z))\geq0\qd\mbox{for}\q r>0,\, z\geq0.$$ In other words, we consider the anti-parallel flow that is non-negative in the upper half-space $\bbR^3_+$. In this case, we identify the non-negative part $\+{\xi}$ of the solution $\xi$ as the restriction of $\xi$ to the half space $\bbR^3_+$. 
    Then the following lemma shows that the impulse of such a solution decreases strictly in time (We do not claim any novelty of this result. For example, it can be found from \cite[Theorem 1.1]{CJ2021}. Still, we give the proof for completeness).

\begin{lem}\label{lem: monoticity of impulse}
    For any nontrivial initial data $\xi_0\in \left(L^1\cap L^\ift\cap L^1_w\right)(\bbR^3)$ satisfying  $r\xi_0\in L^\ift(\bbR^3)$ with the odd symmetry
    $$\xi_0(r,z)=-\xi_0(r,-z)\geq0\qd\mbox{for}\q r>0,\, z\geq0,$$ the impulse of its unique solution $\xi\in L^\ift\left(0,\ift;\left(L^1\cap L^\ift\right)(\bbR^3)\right)$ of \eqref{eq: axi no swirl Euler},
    $$\nrm{r^2\xi(t)}_1=\int_{\bbR^3}(x_1^2+x_2^2)
    |\xi(t)|\,\dd\bfx,$$ decreases strictly in $t\geq0.$
\end{lem}
\begin{proof} 
We note that the solution $\xi(t)$ also has odd symmetry with a sign condition as 
$$\xi(t,(r,z))=-\xi(t,(r,-z))\geq0\qd\mbox{for each}\q r>0,z\geq0.$$ Then we have 
    \begin{equation}\label{eq0312_1}
        \begin{aligned}
        \f{1}{2}\f{d}{dt}\nrm{r^2\xi}_1=\f{d}{dt}\int_{x_3>0}(x_1^2+x_2^2)
    \xi\,\dd\bfx\,&=\,
    \int_{x_3>0}(x_1^2+x_2^2)
    \rd_t\xi\,\dd\bfx\\
    &=\, -\int_{x_3>0}(x_1^2+x_2^2)(\bfu\cdot\nb\xi)\,\dd\bfx\\
    &=\, 2\int_{x_3>0}((x_1,x_2,0)\cdot\bfu)\xi\,\dd\bfx\\
    &=\, 4\pi\int_Q\left(r^2u^r\right)\+{\xi}\dd r\dd z 
    \end{aligned}
    \end{equation} using integration by parts. We recall \eqref{eq: axi. sym. biot-savart law} , which gives 
    $$u^r(r,z)=-\f{\rd_z\calG[\xi](r,z)}{r}.$$
    Here, considering  the odd symmetry of $\xi$, we put 
    $$\begin{aligned}
        \calG[\xi](r,z)\,&=\,\int_\Pi G(r,z,r',z')\xi(r',z')r'\dd r'\dd z'\\
        &=\, \int_Q \left[G(r,z,r',z')-G(r,z,r',-z')\right]\+{\xi(r',z')}r'\dd r'\dd z'.
    \end{aligned}$$
    Then we observe that
    $$G(r,z,r',z')=\f{rr'}{2\pi}\ii{0}{\pi}\f{\cos\vartheta}{\sqrt{r^2+r'^2-2rr'\cos\vartheta+(z-z')^2}}\dd\vartheta=\mathscr{G}(r,r',z-z')$$ for the function $\mathscr{G}(r,r',\cdot):\bbR\to\bbR$ is defined as 
    $$\mathscr{G}(r,r',s):=\f{rr'}{2\pi}\ii{0}{\pi}\f{\cos\vartheta}{\sqrt{r^2+r'^2-2rr'\cos\vartheta+s^2}}\dd\vartheta$$ for fixed $r,r'>0$. Then we rewrite \eqref{eq0312_1} as 
    $$\begin{aligned}
        \f{1}{2}\f{d}{dt}\nrm{r^2\xi}_1\,&=\,
    -4\pi\int_{Q\times Q}
    \left[\left(\rd_s\mathscr{G}\right)(r,r',z-z')-
    \left(\rd_s\mathscr{G}\right)(r,r',z+z')\right]\+{\xi(r,z)}\+{\xi(r',z')}rr'\dd r\dd z \dd r'\dd z'\\&=:\, -4\pi\left[(I)-(II)\right].
    \end{aligned}$$ We observe that the right-hand side is negative given that the following properties hold for fixed $r,r'>0$: \\
    
    (i) $\rd_s\mathscr{G}(r,r',s)$ is an odd function in $s\in\bbR$.\\
    
    (ii) It holds that $\rd_s\mathscr{G}(r,r',s)<0$ for each $s>0$.\\

    \noindent Indeed, we observe that (i) implies  
    $$(I)=\int_{Q\times Q}\left(\rd_s\mathscr{G}\right)(r,r',z-z')\cdot\+{\xi(r,z)}\+{\xi(r',z')}rr'\dd r\dd z \dd r'\dd z'=0$$ by changing the variable $z\leftrightarrow z'$, and that (ii) implies 
    $$(II)=\int_{Q\times Q}\left(\rd_s\mathscr{G}\right)(r,r',z+z')\cdot\+{\xi(r,z)}\+{\xi(r',z')}rr' \dd r\dd z \dd r'\dd z'<0.$$ Then we obtain 
    $$\f{1}{2}\f{d}{dt}\nrm{r^2\xi}_1=-4\pi\left[(I)-(II)\right]<0$$ which completes the proof of Lemma~\ref{lem: monoticity of impulse}. It remains to show the properties (i) and (ii). We first observe that (i) follows directly from the observation that $\mathscr{G}(r,r',s)$ is an even function in $s\in\bbR$. For (ii), we observe that 
    $$\begin{aligned}
        \rd_s\mathscr{G}(r,r',s)\,&=\,\f{rr'}{2\pi}\ii{0}{\pi}\f{-s\cdot\cos\vartheta}{(r^2+r'^2-2rr'\cos\vartheta+s^2)^{3/2}}\dd\vartheta\\
        &=\, \f{rr'}{2\pi}\left[\ii{0}{\pi/2}+\ii{\pi/2}{\pi}\right]\\
        &=\f{rr'}{2\pi}\ii{0}{\pi/2}\left[
        \f{-s\cdot\cos\vartheta}{(r^2+r'^2-2rr'\cos\vartheta+s^2)^{3/2}}+\f{s\cdot\cos\vartheta}{(r^2+r'^2+2rr'\cos\vartheta+s^2)^{3/2}}\right]\dd\vartheta\\
        &=\f{rr'}{2\pi}\ii{0}{\pi/2}\underbrace{(s\cdot\cos\vartheta)}_{>0}\underbrace{\left[
        \f{-1}{(r^2+r'^2-2rr'\cos\vartheta+s^2)^{3/2}}+
        \f{1}{(r^2+r'^2+2rr'\cos\vartheta+s^2)^{3/2}}
        \right]}_{<0}\dd\vartheta<0
    \end{aligned}$$ given that $s>0$. The proof is complete.
\end{proof}

\begin{rmk}\label{rmk: L^ift bound of velocity}
    Lemma~\ref{lem: monoticity of impulse} implies that the velocity $\bfu=\calK[\xi]$ has a uniform bound in space-time, given the odd symmetry of the solution:
    $$\xi(t,(r,z))=-\xi(t,(r,-z))\geq0\qd\mbox{for each }r>0,\,z\geq0.$$ To see this, we refer to an inequality:
    \begin{equation}\label{eq: Reng-Sverak inequality}
        \nrm{\calK[\xi]}_{L^\ift(\bbR^3)}
        \lesssim
        \nrm{\xi}_{L^\ift(\bbR^3)}^{1/2}
        \nrm{\xi}_{L^1(\bbR^3)}^{1/4}
        \nrm{\xi}_{L^1_w(\bbR^3)}^{1/4}
    \end{equation} given in \cite{FS2015}. While $L^1$ and $L^\ift$ norms are conserved:
    $$\nrm{\xi(t)}_{L^\ift(\bbR^3)}=
    \nrm{\xi_0}_{L^\ift(\bbR^3)}\qd\mbox{and}\qd
    \nrm{\xi(t)}_{L^1(\bbR^3)}=
    \nrm{\xi_0}_{L^1(\bbR^3)},$$ Lemma~\ref{lem: monoticity of impulse} says that the quantity 
    $$\nrm{\+{\xi(t)}}_{L^1_w(\bbR^3)}$$ decreases strictly in $t\geq0$ if we additionally assume the odd symmetry:
    $$\xi_0(r,z)=-\xi_0(r,-z)\geq0\qd\mbox{for}\q r>0,\, z\geq0.$$ Then we obtain a uniform bound of velocity:
    \begin{equation}\label{eq: L^ift bound of velocity}
\sup_{t\geq0}\nrm{\bfu(t)}_{L^\ift(\bbR^3)}
=\sup_{t\geq0}\nrm{\calK[\xi(t)]}_{L^\ift(\bbR^3)}
\lesssim\nrm{\xi_0}_{L^\ift(\bbR^3)}^{1/2}\nrm{\xi_0}_{L^1(\bbR^3)}^{1/4}\nrm{\xi_0}_{L^1_w(\bbR^3)}^{1/4}.
    \end{equation}
\end{rmk}

\subsection{Variational problem concerning Hill's spherical vortex}\label{subsec: var. prob. Hills vortex}\q

In this section, we give some results in \cite{Choi2024} about the variational problem concerning Hill's spherical vortex $\xi_H$. We first introduce a scaling family of Hill's vortex.
\begin{rmk}\label{rmk: scaling of Hill's vortex}
    By the scaling invariance of Euler equations, a family of traveling solutions parametrized by $\lmb,a>0$ is obtained from the scaling
$$\xi_{H(\lmb,a)}:=\lmb\q\xi_H(\bfx/a)\qd\mbox{and}\qd W_{H(\lmb,a)}:=\lmb a^2 W_H$$ such that the function 
$\xi_{\lmb,a}=\xi_{\lmb,a}(t,\bfx)$ defined by 
$$\xi_{\lmb,a}(t,\bfx):=\xi_{H(\lmb,a)}\left(\bfx-tW_{H(\lmb,a)}\bfe_z\right)\qd\mbox{for}\q t\geq0,\, \bfx\in\bbR^3$$ solves \eqref{eq: axi no swirl Euler}.
\end{rmk}

For a family of parameters $0<\mu,\nu,\lmb<\ift$, we define some sets of admissible functions 
\begin{equation*}
    \begin{aligned}
        \calP_{\mu,\nu,\lmb}&:=
\left\{\xi\in L^\ift(\bbR^3):\, \xi=\lmb\bfone_A\q\mbox{for some axisymmetric}\, A\subset\bbR^3,\,\f{1}{2}\nrm{r^2\xi}_1=\mu,\, \nrm{\xi}_1\leq\nu\right\},\\
\calP'_{\mu,\nu,\lmb}&:=
\left\{\xi\in L^\ift(\bbR^3):\, \xi:\mbox{axisymmetric,}\,0\leq\xi\leq\lmb,\,\f{1}{2}\nrm{r^2\xi}_1=\mu,\,
\nrm{\xi}_1\leq\nu\right\},\\
\calP''_{\mu,\nu,\lmb}&:=
\left\{\xi\in L^\ift(\bbR^3):\, \xi:\mbox{axisymmetric,}\,0\leq\xi\leq\lmb,\,\f{1}{2}\nrm{r^2\xi}_1\leq\mu,\, \nrm{\xi}_1\leq\nu\right\}
    \end{aligned}
\end{equation*} 
together with the maximal kinetic energies
$$
\calI_{\mu,\nu,\lmb}:= \sup_{\xi\in \calP_{\mu,\nu,\lmb}}E[\xi],\qd
\calI'_{\mu,\nu,\lmb}:=\sup_{\xi\in \calP'_{\mu,\nu,\lmb}}E[\xi],\qd
\calI''_{\mu,\nu,\lmb}:= \sup_{\xi\in \calP''_{\mu,\nu,\lmb}}E[\xi],\qd
$$ and the sets of maximizers 
$$
\begin{aligned}
    \calS_{\mu,\nu,\lmb}&:=\left\{\xi\in\calP_{\mu,\nu,\lmb}:\, E[\xi]=\calI_{\mu,\nu,\lmb}\right\},\qd\\
\calS'_{\mu,\nu,\lmb}&:=\left\{\xi\in\calP'_{\mu,\nu,\lmb}:\, E[\xi]=\calI'_{\mu,\nu,\lmb}\right\},\qd\\
\calS''_{\mu,\nu,\lmb}&:=\left\{\xi\in\calP''_{\mu,\nu,\lmb}:\, E[\xi]=\calI''_{\mu,\nu,\lmb}\right\}.
\end{aligned}
$$ By their definitions, we have the  relations 
\begin{equation}\label{eq0404_1}
    \calP_{\mu,\nu,\lmb}\subset
\calP'_{\mu,\nu,\lmb}\subset
\calP''_{\mu,\nu,\lmb}
\end{equation} and 
$$\calI_{\mu,\nu,\lmb}\leq
\calI'_{\mu,\nu,\lmb}\leq
\calI''_{\mu,\nu,\lmb}.$$ 

Theorem~\ref{thm: larger admissible sets} below says that, despite the set relations
\eqref{eq0404_1}, three admissible sets $\calP_{\mu,\nu,\lmb},\,
\calP'_{\mu,\nu,\lmb},\,\mbox{and}\,
\calP''_{\mu,\nu,\lmb}$ possess the same maximizer set. In other words, every energy maximizer in the largest admissible set $\calP''_{\mu,\nu,\lmb}$ belongs to the smallest admissible set $\calP_{\mu,\nu,\lmb}$. Theorem~\ref{thm: uniqueness of maximizer}, appearing after Theorem~\ref{thm: larger admissible sets}, characterizes the Hill's vortex 
$\xi_{H(\lmb,a)}$ as the unique energy maximizer in $\calP_{\mu,\nu,\lmb}$ under a certain assumption on the triple of parameters $(\mu,\nu,\lmb)$.
\begin{thm}\cite[Theorem 4.2]{Choi2024}\label{thm: larger admissible sets}
    For $0<\mu,\nu,\lmb<\ift$, we have 
    $$\calI_{\mu,\nu,\lmb}=
\calI'_{\mu,\nu,\lmb}=
\calI''_{\mu,\nu,\lmb}\in(0,\ift)
\qd\mbox{and}\qd 
\calS_{\mu,\nu,\lmb}=
\calS'_{\mu,\nu,\lmb}=
\calS''_{\mu,\nu,\lmb}\neq\0.
$$
\end{thm}

\begin{thm}\cite[Theorem 3.2]{Choi2024}\label{thm: uniqueness of maximizer} There exists a constant $\calM_0>0$ such that for any constants $0<\mu,\nu,\lmb<\ift$ satisfying $\mu\nu^{-5/3}\lmb^{2/3}\leq \calM_0$, we have   
$$\calS_{\mu,\nu,\lmb}=\left\{ \xi_{H(\lmb,a)}(\cdot+c\bfe_z):\, c\in\bbR\right\}$$ where $\xi_{H(\lmb,a)}$ is the Hill's vortex for the vortex strength constant $\lmb$ with the radius $a=a(\lmb,\mu)>0$ solving the relation $\mu=(4/15)\pi\lmb a^5$.
\end{thm}

Indeed, the uniqueness result of energy maximizer, given in Theorem~\ref{thm: uniqueness of maximizer} above, implies that the kinetic energy $I_{\mu,\nu,\lmb}$ can be explicitly expressed using given parameters and the kinetic energy of the (unit) Hill's spherical vortex $\xi_H$. 
\begin{rmk}\label{rmk_energy is scaling}
    Theorem~\ref{thm: uniqueness of maximizer} implies that, for any $0<\mu,\nu,\lmb<\ift$ satisfying $\mu\nu^{-5/3}\lmb^{2/3}\leq \calM$, the maximal kinetic energy is explicitly given by
    $$\calI_{\mu,\nu,\lmb}=\left(\f{15}{4\pi}\right)^{7/5}E[\xi_H]\cdot\lmb^{3/5}\mu^{7/5} .$$ 
\end{rmk}

\subsection{Compactness of the set of maximizers}\q

We borrow some estimates concerning the kinetic energy (Lemma~\ref{lem: energy estimates}) and the compactness theorem (Theorem~\ref{thm: compactness of maximizer set}) from \cite{Choi2024}.  Lemma~\ref{lem: energy estimates} will be repeatedly used in this paper. Theorem~\ref{thm: compactness of maximizer set} says that, roughly speaking, any approximate sequence $\xi_n$ converges to a maximizer $\xi\in\calS_{\mu,\nu,\lmb}$ given that the sequence is (nearly) in the admissible set $\calP_{\mu,\nu,\lmb}$ with its kinetic energy $E[\xi_n]$ converging to the maximal kinetic energy $\calI_{\mu,\nu,\lmb}$.

\begin{lem}\cite[Lemma 2.3]{Choi2024}\label{lem: energy estimates}
    For axisymmetric $\xi,\xi_1,\xi_2\in \left(L^1\cap L^2\cap L^1_w\right)(\bbR^3)$, we have 
    \begin{equation}\label{eq: energy bound}
        |E[\xi]|\leq E[|\xi|]\lesssim      \left(\nrm{r^2\xi}_1+\nrm{\xi}_{L^1\cap L^2}\right)\nrm{r^2\xi}_1^{1/2}\nrm{\xi}_1^{1/2},
    \end{equation}
    \begin{equation}\label{eq: interaction energy bound}
        \left|\int_\Pi\int_\Pi
        G(r,z,r',z')\xi_1(r,z)\xi_2(r',z') rr' dr'dz'drdz
        \right|\lesssim
        \left(\nrm{r^2\xi_1}_1+\nrm{\xi_1}_{L^1\cap L^2}\right)\nrm{r^2\xi_2}_1^{1/2}\nrm{\xi_2}_1^{1/2},
    \end{equation}
    \begin{equation}\label{eq: energy difference}
        |E[\xi_1]-E[\xi_2]|\lesssim 
        \left(\nrm{r^2(\xi_1+\xi_2)}_1+\nrm{\xi_1+\xi_2}_{L^1\cap L^2}\right)\nrm{r^2(\xi_1-\xi_2)}_1^{1/2}\nrm{\xi_1-\xi_2}_1^{1/2}.
    \end{equation}
\end{lem}

\begin{thm}\cite[Theorem 3.1]{Choi2024}\label{thm: compactness of maximizer set}
    Let $0<\mu,\nu,\lmb<\ift$. Let $\set{\xi_n}_{n=1}^\ift$ be a sequence of non-negative axisymmetric functions in $\bbR^3$ and let $\set{a_n}_{n=1}^\ift$ be a sequence of positive numbers such that 
    $$a_n\to0\qd\mbox{as}\q n\to\ift,$$
    $$\limsup_{n\to\ift}\nrm{\xi_n}_1\leq\nu,\qd
    \lim_{n\to\ift}\int_{\set{\bfx\in\bbR^3:|\xi_n(\bfx)-\lmb|\geq a_n}}\xi_n\,\dd\bfx=0,\qd
    \lim_{n\to\ift}\f{1}{2}\nrm{r^2\xi_n}_1=\mu,$$
    $$\sup_n\nrm{\xi_n}_2<\ift,\qd \mbox{and}\qd \lim_{n\to\ift}E[\xi_n]=\calI_{\mu,\nu,\lmb}.$$ Then there exist a subsequence $\set{\xi_{n_k}}_{k=1}^\ift,$ a sequence $\set{c_k}_{k=1}^\ift\subset \bbR$, and a function $\xi\in \calS_{\mu,\nu,\lmb}$ such that 
    $$\nrm{r^2(\xi_{n_k}(\cdot+c_k\bfe_z)-\xi)}_1\to0\qd\mbox{as}\q k\to\ift.$$
\end{thm}

\section{Stability: proof of Theorem~\ref{thm: main result with no shift estimate}}\label{sec: stability}

\subsection{Scaled version of Theorem~\ref{thm: main result with no shift estimate}}\label{subsec: Scaled version of orbital stability}\q

Theorem~\ref{thm: main result of mu with no shift estimate} below is a scaled version of Theorem~\ref{thm: main result with no shift estimate}. We will first prove Theorem~\ref{thm: main result of mu with no shift estimate} by using Theorem~\ref{thm: uniqueness of maximizer}, and after that, we will prove Theorem~\ref{thm: main result with no shift estimate} based on Theorem~\ref{thm: main result of mu with no shift estimate} without difficulty in Subsection~\ref{subsec: proof of main result with no shift estimate}. We note that Theorem~\ref{thm: uniqueness of maximizer} characterizes the Hill's spherical vortex as the unique maximizer of the aforementioned variational problem: for any small $\mu>0$,
any maximizer in the set of maximizers $\calS_{\mu,1,1}$ is a $z$-translation of \textit{scaled} Hill's spherical vortex $\xi_{H(1,\kpp_\mu)}$ where the radius $\kpp_\mu$ is determined by $\f{4\pi}{15}\kpp_\mu^5=\mu,$ so that we get $\f{1}{2}\nrm{r^2\xi_H^\mu}_1=\mu.$

For given $\mu>0$, we denote
$$\xi_H^\mu:=\xi_{H(1,\kpp_\mu)},\qd\mbox{and}\qd \nu_\mu:=\nrm{\xi_H^\mu}_1=\f{4\pi}{3}\kpp_\mu^3$$  Theorem~\ref{thm: main result of mu with no shift estimate} below establishes the stability of the scaled Hill's vortex $\xi_H^\mu$ for small $\mu>0$.  

\begin{thm}\label{thm: main result of mu with no shift estimate}
    There exists a constant $\mu_0>0$ such that, for any $\mu\in(0,\mu_0)$ and for each $\eps>0$, there exist $\dlt'=\dlt'(\mu,\eps)>0$ and $d_0'=d_0'(\mu,\eps)>\kpp_\mu$ satisfying the following: \\
    
      For any $d\geq d_0'$ and for any axisymmetric initial data $\xi_0\in \left(L^1\cap L^\ift\cap L^1_w\right)(\bbR^3)$ with $r\xi_0\in L^\ift(\bbR^3)$ such that
     $$\xi_0(r,z)=-\xi_0(r,-z)\geq0\qd\mbox{for}\q r>0,\,z\geq0$$ and 
    $$\nrm{\+{\xi_0}-\xi^\mu_H\left(\cdot-d\bfe_z\right)}_{\left(L^1\cap L^2\cap L^1_w\right)(\bbR^3)}<\dlt',$$ the axisymmetric weak solution $\xi(t)$ of \eqref{eq: axi no swirl Euler} satisfies 
    $$\inf_{\tau\geq\kpp_\mu}\nrm{\+{\xi(t)}-\xi_H^\mu\left(\cdot-\tau\bfe_z\right)}_{\left(L^1\cap L^2\cap L^1_w\right)(\bbR^3)}<\eps\qd\mbox{for each}\q t\geq0.$$
\end{thm}

\begin{proof}[Proof of Theorem~\ref{thm: main result of mu with no shift estimate}]
    We put $\mu_0:=\calM_0$ for the constant $\calM_0>0$ given in Theorem~\ref{thm: uniqueness of maximizer}. Then we suppose the contrary, i.e., there exist $\mu\in(0,\calM_0)$ and $\eps>0$ such that, 
    for any sequences $\dlt_n\searrow0$ and $d_{0,n}\nearrow\ift$, there exist a sequence $d_n\geq d_{0,n}$, a sequence of axisymmetric solutions $\xi_n(t)$, and a sequence $t_n\geq0$ such that 
    $$\nrm{\+{\xi_n(0)}-\xi^\mu_H\left(\cdot-d_n \bfe_z\right)}_{\left(L^1\cap L^2\cap L^1_w\right)(\bbR^3)}<\dlt_n
    \qd\mbox{and}\qd
    \inf_{\tau\geq\kpp_\mu}\nrm{\+{\xi_n(t_n)}-\xi_H^\mu\left(\cdot-\tau\bfe_z\right)}_{\left(L^1\cap L^2\cap L^1_w\right)(\bbR^3)}\geq\eps.$$ 
    We note that the inequality $\mu<\calM_0$ implies $\xi_H^\mu\in \calS_{\mu,1,1}$ by Theorem~\ref{thm: uniqueness of maximizer}, which says that $\nu_\mu<1$. \\
    
    \noindent\textbf{[1] Properties of initial data 
    $\xi_n(0)$ :}
    \begin{equation}\label{eq_0210_1}
        \qd (i)\, \f{1}{2}\nrm{r^2\+{\xi_n(0)}}_1\to\mu,\qd
    (ii)\,\nrm{\+{\xi_n(0)}}_1\to\nu_\mu,\qd\mbox{and}
    \qd(iii)\, E\left[\+{\xi_n(0)}\right]\to \calI_{\mu,1,1}.
    \end{equation}\
    From the assumption on $\+{\xi_n(0)}$, we directly obtain \textit{(i)} and \textit{(ii)}. Moreover, we get \textit{(iii)}: $$E\left[\brk{\xi_n(0)}_+\right]\to E\left[\xi^\mu_H\right]=\calI_{\mu,1,1}\qd\mbox{as}\q n\to\ift$$
due to the estimate \eqref{eq: energy difference} in Lemma~\ref{lem: energy estimates}. We proved \textit{(i)}, \textit{(ii)}, and \textit{(iii)}. \\

    \noindent\textbf{[2] Properties of $\xi_n(t_n)$ :}
    \begin{equation}\label{eq_0210_2}
        \qd (iv)\,\f{1}{2}\nrm{r^2\+{\xi_n(t_n)}}_1\to\mu,\qd
    (v)\,\nrm{\+{\xi_n(t_n)}}_1\to\nu_\mu,\qd\mbox{and}
    \qd(vi)\, E\left[\+{\xi_n(t_n)}\right]\to \calI_{\mu,1,1}.
    \end{equation}

    We obtain \textit{(v)} from \textit{(ii)} in \eqref{eq_0210_1}
    by the conservation of $L^1$ norm, i.e.,
    $$\nrm{\+{\xi_n(t_n)}}_1=\f{1}{2}\nrm{\xi_n(t_n)}_1
    =\f{1}{2}\nrm{\xi_n(0)}_1=
    \nrm{\+{\xi_n(0)}}_1\to\nu_\mu\qd\mbox{as}\q n\to\ift.$$
    Before the proof of \textit{(iv)} and \textit{(vi)}, we explore some properties of kinetic energy $E[\cdot]$:
\begin{itemize}
        \item The total kinetic energy $E[\xi_n(t)]$ is constant in $t\geq0$, i.e.,
    $$\begin{aligned}
        E[\xi_n(t)]=\f{1}{2}\int_{\bbR^3}
        \calG[\xi_n(t)]\xi_n(t)\,\dd\bfx=
        E[\xi_n(0)]\qd\mbox{for all }t\geq0.
    \end{aligned}$$
    \item The total kinetic energy $E[\xi_n(t)]$ has the relation 
    $$E[\xi_n(t)]=2E[\+{\xi_n(t)}]-2E^{inter}_n(t)$$ where the interaction energy $E^{inter}_n$ is given as 
     $$E^{inter}_n(t):=\f{1}{2}\int_{\bbR^3}\calG[\+{\xi_n(t)}]\+{-\xi_n(t)}\,\dd\bfx>0.$$ The above relation follows from the odd symmetry of the solution $\xi_n(t)$ and the property of the Green's function that $G(\bfx,\bfy)=G(\bfy,\bfx)$. Indeed, splitting the solution as 
    $$\xi_n(t)=\+{\xi_n(t)}-\+{-\xi_n(t)},$$
we get the above relation as 
    $$\begin{aligned}
        E[\xi_n(t)]=\f{1}{2}\int_{\bbR^3}
        \calG[\xi_n(t)]\xi_n(t)\,\dd\bfx
        \,&=\,
        \int_{\bbR^3}
        \calG[\+{\xi_n(t)}]
        \+{\xi_n(t)}\,\dd\bfx-
        \int_{\bbR^3}
        \calG[\+{\xi_n(t)}]
        \+{-\xi_n(t)}\,\dd\bfx\\
        &=\, 2E[\+{\xi_n(t)}]-2E^{inter}_n(t).\q 
    \end{aligned}$$ 
    \end{itemize}\noindent    
    We will now prove \textit{(vi)} of \eqref{eq_0210_2} by the observations above (later we will see that \textit{(iv)} of \eqref{eq_0210_2} will be deduced from several results that will be obtained during the proof of \textit{(vi)}). To prove \textit{(vi)}: $$E\left[\+{\xi_n(t_n)}\right]\to \calI_{\mu,1,1}\qd\mbox{as}\q n\to\ift,$$ we observe that
 $$    \begin{aligned}
        |\calI_{\mu,1,1}-E\left[\+{\xi_n(t_n)}\right]|\,&\leq\, \left|\calI_{\mu,1,1}-E\left[\+{\xi_n(0)}\right]\right|+
    \left|E\left[\+{\xi_n(0)}\right]-E\left[\+{\xi_n(t_n)}\right]\right|\\
    &=\,
    \left|\calI_{\mu,1,1}-E\left[\+{\xi_n(0)}\right]\right|+\left|E^{inter}_n(t_n)-E^{inter}_n(0)\right|\\
    &\leq\, \underbrace{\left|\calI_{\mu,1,1}-E\left[\+{\xi_n(0)}\right]\right|}_{\to\q0\q\q\mbox{by }(iii)\q in \q \eqref{eq_0210_1}}+E^{inter}_n(t_n)+E^{inter}_n(0)
    \end{aligned}
    $$ which implies that it suffices to show 
$$E^{inter}_n(0),\,E^{inter}_n(t_n)\to0\qd\mbox{as}\q n\to\ift.$$ We first observe that  
    $E^{inter}_n(0)\to0\qd\mbox{as}\q n\to\ift$ 
    by the assumptions on $\xi_n(0)$ and $d_n$. Indeed, from \eqref{eq: interaction energy bound} in Lemma~\ref{lem: energy estimates}, we obtain $$\lim_{n\to\ift}\left|E^{inter}_n(0)-\f{1}{2}\int_{\bbR^3}\calG[\xi_H(\cdot-d_n\bfe_z)]\xi_H(\bfx+d_n\bfe_z)\,\dd\bfx\right|=0,$$ where    
    $$\left|\int_{\bbR^3}\calG[\xi_H(\cdot-d_n\bfe_z)]\xi_H(\bfx+d_n\bfe_z)\,\dd\bfx\right| = 
    \calO(d_n^{-1})\q\to0 \qd\mbox{as}\q n\to\ift    
    $$ based on the obvious observation of the Green's function $G$:  
    $$|G(r,z,r',z')|\lesssim \f{rr'}{|z-z'|}$$ (see \eqref{eq: Green function formula}).
    For $E^{inter}_n(t_n)$, we observe that 
     \begin{equation}\label{eq_0207_4}
     \begin{aligned}
         0\,\leq\, \limsup_{n\to\ift}E^{inter}_n(t_n)\,&=\, \limsup_{n\to\ift} \left(E\left[\+{\xi_n(t_n)}\right]-\f{1}{2}E[\xi_n(t_n)]\right)\\&=\,
    \limsup_{n\to\ift} \left(E\left[\+{\xi_n(t_n)}\right]-\f{1}{2}E[\xi_n(0)]\right)\\
    &=\, \limsup_{n\to\ift}E\left[\+{\xi_n(t_n)}\right]-\calI_{\mu,1,1}
     \end{aligned}
     \end{equation} which follows from the relation that
    $$E[\xi_n(0)]=2\underbrace{E[\+{\xi_n(0)}]}_{\to\q\calI_{\mu,1,1}}-2\underbrace{E^{inter}_n(0)}_{\to\q0}\to 2\calI_{\mu,1,1}.$$ Therefore, it remains to show that 
    \begin{equation}\label{eq_0207_1}
        \limsup_{n\to\ift} E\left[\+{\xi_n(t_n)}\right]\leq \calI_{\mu,1,1}
    \end{equation} to obtain $E^{inter}_n(t_n)\to0$. We first observe that 
    $$\left|E\left[\+{\xi_n(t_n)}\right]-E\left[\min\left\{1,\+{\xi_n(t_n)}\right\}\right]\right|\to0\qd\mbox{as}\q n\to\ift,$$
    which is deduced from the relation
    $$\begin{aligned}
        \left\|\+{\xi_n(t_n)}-\min\left\{1,\+{\xi_n(t_n)}\right\}\right\|_1
        &=\,\nrm{\+{\xi_n(t_n)-1}}_1\\
        &=\,\nrm{\+{\xi_n(0)-1}}_1\\
        &\leq\,
        \nrm{\+{\xi_n(0)-\xi_H^\mu(\cdot-d_n\bfe_z)}}_1\to0
        \qd\mbox{as}\q n\to\ift
    \end{aligned}$$ and the estimate \eqref{eq: energy difference} in Lemma~\ref{lem: energy estimates}. Then we have 
    \begin{equation}\label{eq_0207_2}
        \limsup_{n\to\ift} E\left[\+{\xi_n(t_n)}\right]=
    \limsup_{n\to\ift} E\left[\min\left\{1,\+{\xi_n(t_n)}\right\}\right].
    \end{equation}
    We define $$\mu_n(t):=\f{1}{2}\nrm{r^2\+{\xi_n(t)}}_1\qd\mbox{for each }t\geq0.$$ Then, for large $n\geq1$ satisfying
    $$\left\|\min\left\{1,\+{\xi_n(t_n)}\right\}\right\|_1\q\q\leq\underbrace{\nrm{\+{\xi_n(t_n)}}_1}_{\to \q\nu_\mu<1\q\mbox{by}\q(v)\q\mbox{of}\q\eqref{eq_0210_2} }<1,$$ we have 
    $$\min\left\{1,\+{\xi_n(t_n)}\right\}\in \calP''_{\mu_n(t_n),1,1}.$$
    It follows that 
    $$E\left[\min\left\{1,\+{\xi_n(t_n)}\right\}\right]\leq\calI''_{\mu_n(t_n),1,1}=\calI_{\mu_n(t_n),1,1}$$ for large $n\geq1$. Combined with \eqref{eq_0207_2}, the above relation implies 
    \begin{equation}\label{eq_0207_3}
        \limsup_{n\to\ift} E\left[\+{\xi_n(t_n)}\right]=
    \limsup_{n\to\ift} E\left[\min\left\{1,\+{\xi_n(t_n)}\right\}\right]\leq 
\limsup_{n\to\ift}\calI_{\mu_n(t_n),1,1}\leq\lim_{n\to\ift}\calI_{\mu_n(0),1,1}=\calI_{\mu,1,1}
    \end{equation} which gives \eqref{eq_0207_1}. Here we used the continuity and monotonicity of $\calI_{(\cdot),1,1}$ in $(0,\mu_0]$ given in Remark~\ref{rmk_energy is scaling}, the monotonicity of $\mu_n(\cdot)$ given in Lemma~\ref{lem: monoticity of impulse}, and the convergences $\mu_n(0)\to\mu$ given in \textit{(i)} of \eqref{eq_0210_1}. It proves \textit{(vi)} of \eqref{eq_0210_2}. 

    It remains to show \textit{(iv)} of \eqref{eq_0210_2}. The convergence \textit{(vi)}:
    $$E\left[\+{\xi_n(t_n)}\right]\to \calI_{\mu,1,1}\qd\mbox{as}\q n\to\ift$$ is applied to \eqref{eq_0207_3} and leads to the convergence 
    $\calI_{\mu_n(t_n),1,1}\to\calI_{\mu,1,1}.$ Hence, by Remark~\ref{rmk_energy is scaling}, we have $$\mu_n(t_n)\to\mu\qd\mbox{as}\q n\to\ift.$$ It proves \textit{(iv)} of \eqref{eq_0210_2}.\\

    \noindent\textbf{[3] Convergence of $\+{\xi_n(t_n)}$ to $\xi_H^\mu$ up to translations: For some $\set{\bt_n}\subset\bbR$, it holds that}
    $$\nrm{\+{\xi_n(t_n)}-\xi_H^\mu(\cdot-\bt_n \bfe_z)}_{\left(L^1\cap L^2\cap L^1_w\right)(\bbR^3)}\to0\qd\mbox{as}\q n\to\ift.$$

    As we have shown $$(iv)\, \f{1}{2}\nrm{r^2\+{\xi_n(t_n)}}_1\to\mu,\q\q
    (v)\,\nrm{\+{\xi_n(t_n)}}_1\to\nu_\mu,\q\q\mbox{and}
    \qd(vi)\,E\left[\+{\xi_n(t_n)}\right)\to \calI_{\mu,1,1}\qd\mbox{as}\q n\to\ift,$$ in the previous step, we will now apply Theorem~\ref{thm: compactness of maximizer set} for the sequence of function $\+{\xi_n(t_n)}$ and $\nu=\lmb=1$, by showing some additional properties on $\+{\xi_n(t_n)}$:
    \begin{equation}\label{eq_0124_1}
        \sup_n\nrm{\+{\xi_n(t_n)}}_2<\ift
    \end{equation}
 and 
    \begin{equation}\label{eq_0124_2}
        \lim_{n\to\ift}\int_{\set{\bfx\in\bbR^3:|\+{\xi_n(t_n)}-1|\geq a_n}}\+{\xi_n(t_n)}\,\dd\bfx=0
        \qd\mbox{for some sequence}\q a_n\searrow0.
    \end{equation} We can easily obtain \eqref{eq_0124_1} from the observation
    $$\nrm{\+{\xi_n(t_n)}}_2=\nrm{\+{\xi_n(0)}}_2\to \nrm{\xi_H^\mu}_2\qd\mbox{as}\q n\to\ift.$$ To show \eqref{eq_0124_2}, we first observe that 
    $$\begin{aligned}
        \limsup_{n\to\ift}\int_{\set{ |\+{\xi_n(t_n)}-1|\geq a_n}}\+{\xi_n(t_n)}\,\dd\bfx\,&=\,\limsup_{n\to\ift}\int_{\set{|\+{\xi_n(0)}-1|\geq a_n}}\+{\xi_n(0)}\,\dd\bfx\\
    &=\,
    \limsup_{n\to\ift}\int_{\set{|\+{\xi_n(0)}-1|\geq a_n}\cap A_n}\+{\xi_n(0)}\,\dd\bfx
    \end{aligned}$$
     for the set $A_n:=\supp \xi_H^\mu(\cdot-d_n \bfe_z)$ and an arbitrary sequence $a_n$, due to the relation 
    $$\begin{aligned}
        \int_{\set{ |\+{\xi_n(0)}-1|\geq a_n}\cap A_n^c}\+{\xi_n(0)}\,\dd\bfx\,&=\,\int_{\set{ |\+{\xi_n(0)}-1|\geq a_n} 
        \cap A_n^c}\left|
        \+{\xi_n(0)}-\xi_H^\mu(\cdot-d_n \bfe_z)\right|\,\dd\bfx\\
        &\leq\, \nrm{\+{\xi_n(0)}-\xi_H^\mu(\cdot-d_n \bfe_z)}_1\to0\qd\mbox{as}\q n\to\ift.
    \end{aligned}$$
    Moreover, we obtain
    $$\begin{aligned}
        \limsup_{n\to\ift}\int_{\set{|\+{\xi_n(0)}-1|\geq a_n}\cap A_n}\+{\xi_n(0)}\,\dd\bfx\,&=\,
        \limsup_{n\to\ift}\int_{\set{|\+{\xi_n(0)}-\xi_H^\mu(\cdot-d_n \bfe_z)|\geq a_n}\cap A_n}\+{\xi_n(0)}\,\dd\bfx\\
        &\leq\, \limsup_{n\to\ift}\int_{A_n}\f{1}{a_n}\left|\+{\xi_n(0)}-\xi_H^\mu(\bfx-d_n \bfe_z)\right|\+{\xi_n(0)}\,\dd\bfx\\
        &\leq\,  \limsup_{n\to\ift}\f{1}{a_n}\left\|\+{\xi_n(0)}-\xi_H^\mu(\cdot-d_n \bfe_z)\right\|_2\cdot\nrm{\+{\xi_n(0)}}_2\\
        &\leq\, \nrm{\xi_H^\mu}_2 \cdot\limsup_{n\to\ift}\f{1}{a_n}\left\|\+{\xi_n(0)}-\xi_H^\mu(\cdot-d_n \bfe_z)\right\|_2=0
    \end{aligned}
    $$ if we put $a_n=\nrm{\+{\xi_n(0)}-\xi_H^\mu(\cdot-d_n \bfe_z)}_2^{1/2}$. It proves \eqref{eq_0124_2}. Therefore we apply Theorem~\ref{thm: compactness of maximizer set} for the sequence $\+{\xi_n(t_n)}$ to obtain a subsequence of $\set{\+{\xi_n(t_n)}}$ (still denoted by $\set{\+{\xi_n(t_n)}}$), a sequence of translation $\set{\bt_n}\subset\bbR$, and a maximizer $\xi\in \calS_{\mu,1,1}$ such that 
    $$\nrm{r^2(\+{\xi_n(t_n,\cdot+\bt_n \bfe_z)}-\xi)}_1\to0\qd\mbox{as}\q n\to\ift.$$ Without loss of generality, from Theorem~\ref{thm: uniqueness of maximizer} we put $\xi=\xi_H^\mu$ to write 
    $$\nrm{r^2(\+{\xi_n(t_n,\cdot+\bt_n \bfe_z)}-\xi_H^\mu)}_1\to0\qd\mbox{as}\q n\to\ift.$$ 
    It remains to show that 
    $$\+{\xi_n(t_n,\cdot+\bt_n \bfe_z)}\to\xi_H^\mu\qd\mbox{in}\q L^1\cap L^2(\bbR^3).$$ For $L^2$ convergence, from \eqref{eq_0124_1} we obtain a weak limit $\tld{\xi}\in L^2(\bbR^3)$ such that 
    $$\+{\xi_n(t_n,\cdot+\bt_n \bfe_z)}\weakto\tld{\xi}\qd\mbox{in}\q L^2(\bbR^3).$$ We easily obtain $\tld{\xi}=\xi_H^\mu$ by observing that, for any bounded set $U\subset\bbR^3$, we have 
    $$\begin{aligned}
        \left|\int_{U}r^2\left(\tld{\xi}-\xi_H^\mu\right)\,\dd\bfx\right|
    \,&=\,\left|\lim_{n\to\ift}\int_{U}r^2\left(\+{\xi_n(t_n,\cdot+\bt_n \bfe_z)}-\xi_H^\mu\right)\,\dd\bfx\right|\\
    &\leq\, \lim_{n\to\ift}\nrm{r^2(\+{\xi_n(t_n,\cdot+\bt_n \bfe_z)}-\xi_H^\mu)}_1=0.
    \end{aligned}
    $$ Then, we obtain the strong $L^2$ convergence
    $$\+{\xi_n(t_n,\cdot+\bt_n \bfe_z)}\to\xi_H^\mu\qd\mbox{in}\q L^2(\bbR^3).$$
    by the convergence of $L^2$ norm 
$ \nrm{\+{\xi_n(t_n)}}_2\to
\nrm{\xi_H^\mu}_2.$
For $L^1$ convergence, we note $\xi_H^\mu=\bfone_{B_\mu}$ to observe that 
    $$\begin{aligned}
        \nrm{\+{\xi_n(t_n,\cdot+\bt_n \bfe_z)}-\xi_H^\mu}_1\,&=\,\nrm{\+{\xi_n(t_n,\cdot+\bt_n \bfe_z)}-\xi_H^\mu}_{L^1(B_\mu)}+\nrm{\+{\xi_n(t_n,\cdot+\bt_n \bfe_z)}}_{L^1(B_\mu^c)}\\
        &=\,\nrm{\+{\xi_n(t_n,\cdot+\bt_n \bfe_z)}-\xi_H^\mu}_{L^1(B_\mu)}+
        \nrm{\+{\xi_n(t_n,\cdot+\bt_n \bfe_z)}}_{L^1(\bbR^3)}-\nrm{\+{\xi_n(t_n,\cdot+\bt_n \bfe_z)}}_{L^1(B_\mu)},
    \end{aligned}
    $$ where 
    $$\lim_{n\to\ift}\nrm{\+{\xi_n(t_n,\cdot+\bt_n \bfe_z)}-\xi_H^\mu}_{L^1(B_\mu)}\leq \lim_{n\to\ift}
    \nrm{\+{\xi_n(t_n,\cdot+\bt_n \bfe_z)}-\xi_H^\mu}_{L^2(\bbR^3)}\cdot\underbrace{|B_\mu|^{1/2}}_{=\nrm{\xi_H^\mu}_2} =0
    ,$$ 
    $$\begin{aligned}
        \lim_{n\to\ift}\nrm{\+{\xi_n(t_n,\cdot+\bt_n \bfe_z)}}_{L^1(B_\mu)}=
    \lim_{n\to\ift}\nrm{\+{\xi_n(t_n,\cdot+\bt_n \bfe_z)}\cdot\xi_H^\mu}_{L^1(\bbR^3)}    &=\lim_{n\to\ift}\int_{\bbR^3}\+{\xi_n(t_n,\cdot+\bt_n \bfe_z)}\cdot\xi_H^\mu\,\dd\bfx\\
    &= \nrm{\xi_H^\mu}_{L^2(\bbR^3)}^2=\nrm{\xi_H^\mu}_{L^1(\bbR^3)}=\nu_\mu,
    \end{aligned}$$ and 
    $$\lim_{n\to\ift}\nrm{\+{\xi_n(t_n,\cdot+\bt_n \bfe_z)}}_{L^1(\bbR^3)}=
    \lim_{n\to\ift}\nrm{\+{\xi_n(0)}}_{L^1(\bbR^3)}=\nrm{\xi_H^\mu}_{L^1(\bbR^3)}=\nu_\mu.$$
    Therefore, we obtain the $L^1$ convergence 
    $$\+{\xi_n(t_n,\cdot+\bt_n \bfe_z)}\to\xi_H^\mu\qd\mbox{in}\q L^1(\bbR^3).$$      

    \noindent\textbf{[4] Lower bound of translation $\set{\bt_n}\subset\bbR$:  
\begin{equation}\label{eq_0219_4}
        \liminf_{n\to\ift} \bt_n\geq \kpp_\mu
    \end{equation} where $\kpp_\mu>0$ is the radius of the (scaled) Hill's spherical vortex $\xi_H^\mu=\xi_{H(1,\kpp_\mu)}$.}\\
    
    Up to now, we have shown the convergence 
    $$\nrm{\+{\xi_n(t_n)}\to\xi_H^\mu(\cdot-\bt_n \bfe_z))}_{\left(L^1\cap L^2\cap L^1_w\right)(\bbR^3)}\to0\qd\mbox{as}\q n\to\ift$$ for some sequence $\set{\bt_n}\subset\bbR.$
    Then we observe that
$$\begin{aligned}
0<\int_{x_3<-\bt_n}\xi_H^\mu(\bfx)\,\dd\bfx\,&=
\int_{x_3<0}\xi_H^\mu(\bfx-\bt_n \bfe_z)\,\dd\bfx\\
&=\int_{x_3<0}\left|\+{\xi_n(t_n)}-\xi_H^\mu(\bfx-\bt_n \bfe_z)\right|\dd\bfx\\
&\leq\,
    \nrm{\+{\xi_n(t_n)}-\xi_H^\mu(\cdot-\bt_n \bfe_z)}_{L^1(\bbR^3)}
    \to0\qd\mbox{as}\q n\to\ift.
\end{aligned}$$ It implies that 
$$0=\limsup_{n\to\ift}\int_{x_3<-\bt_n}\xi_H^\mu\,\dd\bfx=
\int_{x_3<-\liminf_{n\to\ift}{\bt_n}}\xi_H^\mu\,\dd\bfx,$$ and therefore we get \eqref{eq_0219_4}: $$\liminf_{n\to\ift} \bt_n\geq \kpp_\mu.$$

\noindent\textbf{[5] Completion of proof by deriving a contradiction to our initial assumption: }
\begin{equation}\label{eq0319_1}
    \inf_{\tau\geq\kpp_\mu}\nrm{\+{\xi_n(t_n)}-\xi_H^\mu\left(\cdot-\tau\bfe_z\right)}_{\left(L^1\cap L^2\cap L^1_w\right)(\bbR^3)}\geq\eps\qd\mbox{for some}\q \eps>0.
\end{equation}

We recall our previous results: 
    $$\lim_{n\to\ift}\nrm{\+{\xi_n(t_n)}\to\xi_H^\mu(\cdot-\bt_n \bfe_z))}_{\left(L^1\cap L^2\cap L^1_w\right)(\bbR^3)}=0
    \qd\mbox{with}\qd
    \liminf_{n\to\ift}\bt_n\geq\kpp_\mu.$$ 
We first note that, if  $\displaystyle\liminf_{n\to\ift}\bt_n\gneqq\kpp_\mu,$ then 
    $$\bt_n>\kpp_\mu\qd\mbox{for large}\q n\geq1,$$ and it obviously contradicts our initial assumption \eqref{eq0319_1} by
    $$\eps\leq \inf_{\tau\geq\kpp_\mu}\nrm{\+{\xi_n(t_n)}-\xi_H^\mu\left(\cdot-\tau\bfe_z\right)}_{\left(L^1\cap L^2\cap L^1_w\right)(\bbR^3)}\leq \nrm{\+{\xi_n(t_n)}\to\xi_H^\mu(\cdot-\bt_n \bfe_z))}_{\left(L^1\cap L^2\cap L^1_w\right)(\bbR^3)}\to0
    \qd\mbox{as}\q n\to\ift.$$ Now, we assume that    $$\liminf_{n\to\ift}\bt_n=\kpp_\mu.$$
    We now give a claim below which follows from our initial assumption \eqref{eq0319_1}.\\

    \noindent\textit{(Claim)} There exists $s>0$ such that 
    \begin{equation}\label{eq0319_2}
        \inf_{\tau\geq \kpp_\mu-s}\nrm{\+{\xi_n(t_n)}-\xi_H^\mu\left(\cdot-\tau\bfe_z\right)}_{\left(L^1\cap L^2\cap L^1_w\right)(\bbR^3_+)}\geq\f{1}{2}\eps.
    \end{equation}

   For fixed $n\geq1$ we put
    $$\calF(\tau):=\nrm{\+{\xi_n(t_n)}-\xi_H^\mu\left(\cdot-\tau\bfe_z\right)}_{\left(L^1\cap L^2\cap L^1_w\right)(\bbR^3_+)}\qd\mbox{for}\q\tau\in\bbR$$ and observe the continuity
    $$|\calF(\tau_1)-\calF(\tau_2)|\leq \nrm{\xi_H^\mu\left(\cdot-\tau_1\bfe_z\right)-\xi_H^\mu\left(\cdot-\tau_2\bfe_z\right)}_{\left(L^1\cap L^2\cap L^1_w\right)(\bbR^3_+)}\leq C_0|\tau_1-\tau_2|^{1/2}\qd\mbox{for}\q
    |\tau_1-\tau_2|\ll1$$ for some constant $C_0>0$ depending only on $\mu>0$. It gives
    $$
    \inf_{\tau\geq\kpp_\mu}\calF(\tau)\leq 
    \inf_{\tau\geq\kpp_\mu-s}\calF(\tau)+C_0s^{1/2}\qd\mbox{for  }s\ll1,$$ and combining it with \eqref{eq0319_1} leads to 
    $$\begin{aligned}
        \eps\leq\inf_{\tau\geq \kpp_\mu}\nrm{\+{\xi_n(t_n)}-\xi_H^\mu&\left(\cdot-\tau\bfe_z\right)}_{\left(L^1\cap L^2\cap L^1_w\right)(\bbR^3_+)}\,=\,
    \inf_{\tau\geq\kpp_\mu}\calF(\tau)\leq 
    \inf_{\tau\geq\kpp_\mu-s}\calF(\tau)+C_0s^{1/2}\\
    &=\, \inf_{\tau\geq \kpp_\mu-s}\nrm{\+{\xi_n(t_n)}-\xi_H^\mu\left(\cdot-\tau\bfe_z\right)}_{\left(L^1\cap L^2\cap L^1_w\right)(\bbR^3_+)}+C_0s^{1/2}.
    \end{aligned}$$ We now fix small $s>0$ satisfying $C_0s^{1/2}<\eps/2$ to obtain \eqref{eq0319_2}:
    $$\f{1}{2}\eps\leq \inf_{\tau\geq \kpp_\mu-s}\nrm{\+{\xi_n(t_n)}-\xi_H^\mu\left(\cdot-\tau\bfe_z\right)}_{\left(L^1\cap L^2\cap L^1_w\right)(\bbR^3_+)}.$$ It completes the proof of \textit{(Claim)}.\\

    Let $s>0$ be the constant given in the inequality \eqref{eq0319_2} in \textit{(Claim)} above. For such $s>0$, it holds that 
    $$\bt_n>\kpp_\mu-s\qd\mbox{for large}\q n\geq1,$$ and combining it with our previous result \eqref{eq0319_2} gives
    $$\begin{aligned}
        \f{1}{2}\eps\,&\leq\, \inf_{\tau\geq \kpp_\mu-s}\nrm{\+{\xi_n(t_n)}-\xi_H^\mu\left(\cdot-\tau\bfe_z\right)}_{\left(L^1\cap L^2\cap L^1_w\right)(\bbR^3)}\\
        &\leq\, \nrm{\+{\xi_n(t_n)}-\xi_H^\mu(\cdot-\bt_n \bfe_z)}_{L^1(\bbR^3)}\to0\qd\mbox{as}\q n\to\ift
    \end{aligned}$$ which is a contradiction to $\eps>0$.
    It completes the proof of Theorem~\ref{thm: main result of mu with no shift estimate}.
\end{proof} 

\subsection{Proof of Theorem~\ref{thm: main result with no shift estimate}}\label{subsec: proof of main result with no shift estimate}\q

\begin{proof}[Proof of Theorem~\ref{thm: main result with no shift estimate}]
    Let $\eps>0$ be any constant. We fix a small $\mu:=\mu_0/2\in(0,\mu_0)$ for the constant $\mu_0>0$ given in Theorem~\ref{thm: main result of mu with no shift estimate} to guarantee
    $\kpp_\mu<1$ (if necessary, we redefine $\mu_0$ small enough). We define $\dlt,d_0>0$ as 
    $$\dlt=\dlt(\eps):=\kpp_\mu^{-3/2}\dlt'(\mu,\kpp_\mu^5\eps)\qd\mbox{and}\qd 
    d_0=d_0(\eps):=\kpp_\mu^{-1} d_0'(\mu,\kpp_\mu^5\eps)>1$$ for the family of constants $\dlt',d_0'>0$ given in Theorem~\ref{thm: main result of mu with no shift estimate}. We now choose any constant $d\geq d_0$ and any axisymmetric initial data $\xi_0\in\left(L^1\cap L^2\cap L^1_w\right)(\bbR^3)$ having odd symmetry
    $\xi_0(r,z)=-\xi_0(r,-z)\geq0$ for $z\geq0$ and satisfying the assumption 
    $$\nrm{\+{\xi_0}-\xi_H(\cdot-d\bfe_z)}_{\left(L^1\cap L^2\cap L^1_w\right)(\bbR^3)}<\dlt.$$ Then by the scaling on $\xi_0$ given by 
    \begin{equation}\label{eq_0210_3}
        \xi^\mu_0(r,z):=\xi_0\left(\f{r}{\kpp_\mu},\f{z}{\kpp_\mu}\right)
    \end{equation}
    we obtain that 
    $$\begin{aligned}
        \nrm{\+{\xi^\mu_0}-\xi^\mu_H(\cdot-\kpp_\mu d\bfe_z)}_{L^1}\,&=\,
    \kpp_\mu^3\cdot\nrm{\+{\xi_0}-\xi_H(\cdot-d\bfe_z)}_{L^1},\\
    \nrm{\+{\xi^\mu_0}-\xi^\mu_H(\cdot-\kpp_\mu d\bfe_z)}_{L^2}\,&=\,
    \kpp_\mu^{3/2}\cdot\nrm{\+{\xi_0}-\xi_H(\cdot-d\bfe_z)}_{L^2},
    \end{aligned}$$ 
    and $$\nrm{\+{\xi^\mu_0}-\xi^\mu_H(\cdot-\kpp_\mu d\bfe_z)}_{L^1_w}=
    \kpp_\mu^5\cdot\nrm{\+{\xi_0}-\xi_H(\cdot-d\bfe_z)}_{L^1_w}$$ which lead to 
    $$\nrm{\+{\xi^\mu_0}-\xi^\mu_H(\cdot-\kpp_\mu d\bfe_z)}_{L^1\cap L^2\cap L^1_w}\leq
    \kpp_\mu^{3/2}\cdot\nrm{\+{\xi_0}-\xi_H(\cdot-d\bfe_z)}_{L^1\cap L^2\cap L^1_w}<\kpp_\mu^{3/2}\dlt=\dlt'(\mu,\kpp_\mu^5\eps).$$ Together with the observation that 
    $$\kpp_\mu d\geq \kpp_\mu d_0= d_0'(\mu,\kpp_\mu^5\eps),$$ Theorem~\ref{thm: main result of mu with no shift estimate} says that the axisymmetric solution of \eqref{eq: axi no swirl Euler} $\xi^\mu(t)$ with the initial data $\xi_0^\mu$ satisfies 
    $$\inf_{\tau\geq\kpp_\mu}\nrm{\+{\xi^\mu(t)}-\xi_H^\mu(\cdot-\tau\bfe_z)}_{\left(L^1\cap L^2 \cap L^1_w\right)(\bbR^3)}<\kpp_\mu^5\eps\qd\mbox{for each}\q t\geq0.$$ 
    Then by denoting the solution originated from $\xi_0$ as $\xi(t)$, we have the relation 
    $$\xi(t,(r,z))=\xi^\mu\left(\kpp_\mu t,(\kpp_\mu r,\kpp_\mu z)\right)$$ (implying \eqref{eq_0210_3} for $t=0$) and therefore we get
    $$\begin{aligned}
        \nrm{\+{\xi(t)}-\xi_H(\cdot-\tau\bfe_z)}_{L^1}\,&=\, \kpp_\mu^{-3}\cdot\nrm{\+{\xi^\mu(\kpp_\mu t)}-\xi_H^\mu(\cdot-\kpp_\mu \tau\bfe_z)}_{L^1},\\
        \nrm{\+{\xi(t)}-\xi_H(\cdot-\tau\bfe_z)}_{L^2}\,&=\, \kpp_\mu^{-3/2}\cdot\nrm{\+{\xi^\mu(\kpp_\mu t)}-\xi_H^\mu(\cdot-\kpp_\mu\tau\bfe_z)}_{L^2},
    \end{aligned}$$ 
    and 
    $$\nrm{\+{\xi(t)}-\xi_H(\cdot-\tau\bfe_z)}_{L^1_w}\,=\, \kpp_\mu^{-5}\cdot\nrm{\+{\xi^\mu(\kpp_\mu t)}-\xi_H^\mu(\cdot-\kpp_\mu \tau\bfe_z)}_{L^1_w}$$  which lead to 

    $$\begin{aligned}
        \inf_{\tau\geq1}\nrm{\+{\xi(t)}-\xi_H(\cdot-\tau\bfe_z)}_{L^1\cap L^2\cap L^1_w}\,&\leq\, 
    \kpp_\mu^{-5}\cdot\inf_{\tau\geq1}\nrm{\+{\xi^\mu(\kpp_\mu t)}-\xi_H^\mu(\cdot-\kpp_\mu \tau\bfe_z)}_{L^1\cap L^2\cap L^1_w}\\
    &=\, \kpp_\mu^{-5}\cdot\inf_{\tau\geq \kpp_\mu}\nrm{\+{\xi^\mu(\kpp_\mu t)}-\xi_H^\mu(\cdot-\tau\bfe_z)}_{L^1\cap L^2\cap L^1_w}\\
    &<\, \kpp_\mu^{-5}\cdot \kpp_\mu^5\eps=\eps\qd\mbox{for each}\q t\geq0.
    \end{aligned}
    $$ Therefore there exists a shift function $\tau:[0,\ift)\to[1,\ift)$ such that 
$$\nrm{\+{\xi(t)}-\xi_H(\cdot-\tau(t)\bfe_z)}_{L^1\cap L^2\cap L^1_w}<\eps\qd\mbox{for each}\q t\geq0.$$
    It completes the proof of Theorem~\ref{thm: main result with no shift estimate}.
\end{proof}

\section{Shift estimate: proof of Theorem~\ref{thm: shift estimate}}\label{sec: shift}

\subsection{Proof of Theorem  ~\ref{thm: shift estimate}}\label{subsec: shift function estimate}\q

For the solution $\xi(t)$ to \eqref{eq: axi no swirl Euler} of any axisymmetric initial data $\xi_0\in \left(L^1\cap L^\ift\cap L^1_w\right)(\bbR^3)$ with $r\xi_0\in L^\ift(\bbR^3)$ having the odd symmetry 
$$\xi_0(r,z)=-\xi_0(r,-z)\geq0\qd\mbox{for}\q z\geq0,$$ we define its center of mass on the positive part of $x_3$-axis: 
$$z_c(t):= \f{1}{\nrm{\+{\xi(t)}}_{L^1(\bbR^3)}}\int_{x_3>0}x_3\+{\xi(t)}\,\dd\bfx.$$ 
We will estimate the shift function $\tau(\cdot)$ given in Theorem~\ref{thm: main result with no shift estimate} by the bootstrap method based on the estimates of functions $z_c(\cdot)$ and the correlation with the shift function $\tau(\cdot)$ and the center $z_c(\cdot)$.
We will give a lemma that directly proves Theorem  ~\ref{thm: shift estimate}. The lemma below contains the same assumptions as Theorem  ~\ref{thm: shift estimate}, and the estimates given in the conclusion will be used to prove the estimate \eqref{(revisited) eq: tau(t) vs Wt} in Theorem  ~\ref{thm: shift estimate}.

\begin{lem}\label{lem: c.o.m estimate} 
For any $\calM>0$, there exists a constant $\eps_0=\eps_0(\calM)>0$ such that the following holds:\\

For each $\eps\in(0,\eps_0)$, let $d_0=d_0(\eps)>1$ be the corresponding constant given in Theorem~\ref{thm: main result with no shift estimate} (If necessary, we redefine
$d_0(\eps)>\eps^{-1}$). Let $d\geq d_0$ be any constant and 
$\xi_0$ be any axisymmetric initial data satisfying the assumptions in Theorem~\ref{thm: main result with no shift estimate} such that 
$\nrm{\xi_0}_{L^\ift(\bbR^3)}\leq \calM$ and the support of $\xi_0$ is contained in  $\set{\bfx\in\bbR^3:\, d-2\leq|x_3|\leq d+2}.$ Then any shift function $\tau:[0,\ift)\to[1,\ift)$ in Theorem~\ref{thm: main result with no shift estimate} satisfying $\tau(0)=d$ and the center of mass (on $z$-axis) $z_c:[0,\ift)\to\bbR$ of the solution $\xi(t)$ for the initial data $\xi_0$ satisfy the following estimates: 
\begin{equation}\label{eq: z(t) vs tau(t)}
    |z_c(t)-\tau(t)| < C\eps(t+1)\qd\mbox{for each}\q t\geq0
\end{equation} for some constant $C=C(\calM)>0$, and
\begin{equation}\label{eq: z(t) vs Wt}
    |z_c(t)-z_c(0)-W_Ht|< C'\eps t
    \qd\mbox{for each}\q t\geq0
\end{equation} for some uniform constant $C'>0$, where $W_H=2/15$ is the traveling speed of the (single) Hill's spherical vortex $\xi_H$.
\end{lem} 

Lemma~\ref{lem: c.o.m estimate} above will be proved later in Subsection~\ref{subsec: c.o.m estimate}. For now, we prove Theorem  ~\ref{thm: shift estimate} by obtaining the shift estimate \eqref{(revisited) eq: tau(t) vs Wt} from the estimates \eqref{eq: z(t) vs tau(t)} and \eqref{eq: z(t) vs Wt} given in Lemma~\ref{lem: c.o.m estimate} above.

\begin{proof}[Proof of Theorem  ~\ref{thm: shift estimate}] 
The proof follows directly from the estimate \eqref{eq: z(t) vs tau(t)} and \eqref{eq: z(t) vs Wt}. That is, we obtain the shift function estimate \eqref{(revisited) eq: tau(t) vs Wt} in Theorem~\ref{thm: shift estimate} as 
    $$\begin{aligned}
        |\tau(t)-\tau(0)-W_Ht|
        \,&\leq\, |\tau(t)-z_c(t)|+|\tau(0)-z_c(0)|+|z_c(t)-z_c(0)-W_Ht|\\
        &\leq\, C\eps(t+1)+C\eps+
        C'\eps t\\
        &\lesssim_{\calM}\eps(t+1).
    \end{aligned}$$It completes the proof.
\end{proof}

\subsection{Proof of Lemma~\ref{lem: c.o.m estimate}}\label{subsec: c.o.m estimate}

\begin{proof}[Proof of Lemma~\ref{lem: c.o.m estimate}] From the assumption 
\begin{equation}\label{eq0718_3}
    \nrm{\+{\xi(t)}-\xi_H(\cdot-\tau(t)\bfe_z)}_{L^1\cap L^2\cap L^1_w(\bbR^3)}<\eps,
\end{equation}
 we may assume that $\eps>0$ is small enough to give
\begin{equation}\label{eq0307_3}
    \f{1}{2}\nrm{\xi_H}_{L^p}\leq\nrm{\+{\xi_0}}_{L^p}=\nrm{\+{\xi(t)}}_{L^p}\leq 2\nrm{\xi_H}_{L^p}\qd\mbox{for}\q p\in\set{1,2}
\end{equation} and 
\begin{equation}\label{eq0307_4}
    \f{1}{2}\nrm{\xi_H}_{L^1_w}\leq\nrm{\+{\xi(t)}}_{L^1_w}\leq 2\nrm{\xi_H}_{L^1_w}.
\end{equation} In other words, $L^1$, $L^2$, and $L^1_w$ norms of the solution $\+{\xi(t)}$ will be regarded as uniform constants hereafter. We first prove that the estimate \eqref{eq: z(t) vs tau(t)} holds for $t=0$. Observe that \eqref{eq0307_3} and the initial assumption 
\begin{equation}\label{eq0319_3}
    \supp\xi_0\subset\set{\bfx\in\bbR^3:\, 
    d-2\leq |x_3|\leq d+2}
\end{equation}
give us
$$\begin{aligned}
    |z_c(0)-\tau(0)|\,&=\,\f{1}{\nrm{\+{\xi_0}}_{L^1}}\left|\int_{\bbR^3}\left(x_3-\tau(0)\right)\+{\xi_0}\,\dd\bfx\right|\\
    &=\, \f{1}{\nrm{\+{\xi_0}}_{L^1}}\left|\int_{\bbR^3}\left(x_3-\tau(0)\right)\left(\+{\xi_0}-\xi_H(\cdot-\tau(0)\bfe_z)\right)\,\dd\bfx\right|\\
    &\leq\, \f{1}{\nrm{\+{\xi_0}}_{L^1}}\int_{\bbR^3}|x_3-\tau(0)|\cdot\left|\+{\xi_0}-\xi_H(\cdot-\tau(0)\bfe_z)\right|\,\dd\bfx\\
    &\leq\, \f{2}{\nrm{\+{\xi_0}}_{L^1}}\eps\,\leq\, \f{4}{\nrm{\xi_H}_{L^1}}\eps
    <\eps
\end{aligned}$$ by using $\nrm{\xi_H}_{L^1}=4\pi/3$. Here, we used the relation $\int_{\bbR^3}x_3\xi_H\dd\bfx=0.$ In other words, we get
\begin{equation}\label{eq0307_1}
    |z_c(0)-d|<\eps
\end{equation} by $\tau(0)=d$. The remaining proof consists of 3 steps. During the proof, we will constantly assume the smallness of $\eps>0$ depending only on $\calM>0$.\\

\noindent\textbf{[Step 1] The estimate \eqref{eq: z(t) vs tau(t)} holds: 
$$|z_c(t)-\tau(t)|< C_0\eps(1+t)\qd\mbox{for each}\q t\geq0$$ for some $C_0=C_0(\calM)>0$.}

     By observing that $$0=\int_{\bbR^3}x_3\xi_H(\bfx)\,\dd\bfx,
    $$ we have
    \begin{equation}\label{eq0306_1}
        \begin{aligned}
        |z_c(t)-\tau(t)|
        \,&=\,\f{1}{\nrm{\+{\xi_0}}_{L^1}}
    \left|\int_{\bbR^3}(x_3-\tau(t))\+{\xi(t)}\,\dd\bfx\right|\\
    &\lesssim\, \left|\int_{\bbR^3}(x_3-\tau(t))\+{\xi(t)}\,\dd\bfx\right|\\
    &=\,
    \left|\int_{\bbR^3}(x_3-\tau(t))\left[\+{\xi(t)}-\xi_H(\cdot-\tau(t)\bfe_z)\right]\,\dd\bfx\right|\\
    &\leq\, \left|\int_{\supp\+{\xi(t)}}\right|+\left|\int_{\supp\xi_H(\cdot-\tau(t)\bfe_z)\setminus\supp\+{\xi(t)}}\right|\\
    &=:\, (I)+(II).
    \end{aligned}
    \end{equation}
    We first estimate $(II)$ as  
    $$\begin{aligned}
        (II)\,&\leq\,\int_{\supp\xi_H(\cdot-\tau(t)\bfe_z)\backslash\supp\+{\xi(t)}}\underbrace{|x_3-\tau(t)|}_{\leq1}\cdot\left|\+{\xi(t)}-\xi_H(\cdot-\tau(t)\bfe_z)\right|\,\dd\bfx\\
        &\leq\,
    \nrm{\+{\xi(t)}-\xi_H(\cdot-\tau(t)\bfe_z)}_{L^1}
    < \eps.
    \end{aligned} $$
For $(I)$, we similarly put  
$$\begin{aligned}
    (I)\,&\leq\, \int_{\supp\+{\xi(t)}}|x_3-\tau(t)|\cdot\left|\+{\xi(t)}-\xi_H(\cdot-\tau(t)\bfe_z)\right|\,\dd\bfx\\
        &\leq\,
    \nrm{\+{\xi(t)}-\xi_H(\cdot-\tau(t)\bfe_z)}_{L^1}\cdot
    \left(\sup_{\bfx\in\supp\+{\xi(t)}}|x_3-\tau(t)|\right)\\
    &<\, \eps\cdot\calI(t)
\end{aligned}$$ 
where $$\calI(t)\q:=\sup_{\bfx\in\supp\+{\xi(t)}}|x_3-\tau(t)|$$
Note that the assumption \eqref{eq0319_3} implies $|\calI(0)|\leq2.$ 
Then we have 
    \begin{equation}\label{eq0306_4}
        |\calI(t)|\leq |\calI(t)-\calI(0)|+2
    \end{equation}
    where
    $$\begin{aligned}
        |\calI(t)-\calI(0)|\,&=\, \left|\,\left(\sup_{\bfx\in\supp\+{\xi_0}}|\phi(t,\bfx)\cdot\bfe_z-\tau(t)|\right)-\left(\sup_{\bfx\in\supp\+{\xi_0}}|x_3-\tau(0)|\right)\,\right|\\
        &\leq\, 
        \left(\sup_{\bfx\in\supp\+{\xi_0}}
        |\phi(t,\bfx)\cdot\bfe_z-x_3|\right)+|\tau(t)-\tau(0)|\\
        &\leq\, 
        \left(\sup_{\bfx\in\supp\+{\xi_0}}
        |\phi(t,\bfx)-\bfx|\right)+|\tau(t)-\tau(0)|
    \end{aligned}$$ where $\phi(t,\cdot)$ denotes the particle-trajectory map. Then, for any $\bfx\in\bbR^3,$ it holds that
    \begin{equation}\label{eq0306_2}
        |\phi(t,\bfx)-\bfx|\leq t\cdot\sup_{s\geq0}\nrm{\bfu(s)}_{L^\ift(\bbR^3)}\leq t\cdot C_1\nrm{\xi_0}_{L^\ift(\bbR^3)}^{1/2}
        \nrm{\xi_0}_{L^1(\bbR^3)}^{1/4}
        \nrm{r^2\xi_0}_{L^1(\bbR^3)}^{1/4}\lesssim \calM^{1/2}t
    \end{equation} for some uniform constant $C_1>0$ (Here, we used the estimate \eqref{eq: L^ift bound of velocity} in Remark~\ref{rmk: L^ift bound of velocity} to bound 
    $\sup_{s\geq0}\nrm{\bfu(s)}_{L^\ift(\bbR^3)}$, and the last inequality follows from the estimates  \eqref{eq0307_3}, \eqref{eq0307_4}).
    Therefore we get
    $$\begin{aligned}
        \calI(t)\lesssim
        1+\calM^{1/2}t+|\tau(t)-\tau(0)|
    \end{aligned}$$
    In sum, from \eqref{eq0306_1}, we arrive at    
        \begin{equation}\label{eq0306_3}
            |z_c(t)-\tau(t)|\lesssim
            (I)+(II)\leq
    \eps\calI(t)+\eps
\lesssim\eps\left(1+\calM^{1/2}t+|\tau(t)-\tau(0)|\right).
        \end{equation} It remains to estimate $|\tau(t)-\tau(0)|$ in the right-hand side. As we have $\xi_H(\cdot-\tau(t)\bfe_z)\simeq \+{\xi(t)}$, there exists $\bfy^t\in\bbR^3$ such that 
        $$\bfy^t\q\in\q \supp\xi_H\left(\cdot-\tau(t)\bfe_z\right)\cap \supp\+{\xi(t)}$$ which gives $|\tau(t)-y_3|\leq1.$ On the other hand, by the assumption \eqref{eq0319_3}, we observe that the starting point $\bfy^{0,t}\in\supp\+{\xi_0}$ of the point $\bfy^t\in\supp\+{\xi(t)}$ with the relation
        $\bfy^t=\phi(t,\bfy^{0,t})$
        satisfies $$|\tau(0)-y_3^{0,t}|\leq2.$$ Then we have 
        $$\begin{aligned}
            |\tau(t)-\tau(0)|\,&\leq\, |\tau(t)-y_3^t|+|\tau(0)-y_3^{0,t}|+|y_3^t-y_3^{0,t}|\\
            &\leq\, 1+2+|\phi(t,\bfy^{0,t})-\bfy^{0,t}|\\
            &\leq\, 3+t\cdot\sup_{s\geq0}\nrm{\bfu(s)}_{L^\ift(\bbR^3)},
        \end{aligned}
        $$ and by the same argument with \eqref{eq0306_2}, we get 
        $$|\tau(t)-\tau(0)|\lesssim 1+\calM^{1/2}t.$$
We plug the estimate above into \eqref{eq0306_3} and obtain the estimate \eqref{eq: z(t) vs tau(t)}: 
$$|z_c(t)-\tau(t)|\lesssim \eps \left(1+\calM^{1/2}t\right).$$

\noindent\textbf{[Step 2] It holds that}
\begin{equation}\label{eq_0214_1}
    \left|\f{d}{dt}z_c(t)-W_H\right|< C_2\left(\eps+\f{1}{\tau(t)-1}\right)\qd\mbox{if}\,
    \tau(t)\geq2
\end{equation} \textbf{for some uniform constant $C_2>0$.}

We observe that 
\begin{equation}\label{eq_0213_2}
    \begin{aligned}
    \f{d}{dt}z_c(t)\,&=\,\nrm{\+{\xi}}_{L^1}^{-1}
    \int_{x_3>0}x_3\rd_t\xi\,\dd\bfx\\
    &=\,\nrm{\+{\xi}}_{L^1}^{-1}
    \int_{x_3>0}\xi u^z\,\dd\bfx\\
    &=\,\nrm{\+{\xi}}_{L^1}^{-1}
    \int_{\bbR^3}
    \+{\xi}\left(\calK[\xi]\cdot\bfe_z\right)\,\dd\bfx\\
\end{aligned}
\end{equation} Here, we define a bilinear operator $\calB$ as
$$\calB[\xi_1,\xi_2]:= \int_{\bbR^3}\xi_1\left(\calK[\xi_2]\cdot\bfe_z\right)\,\dd\bfx\qd\mbox{for axisymmetric}\q \xi_1,\xi_2\in \left(L^1\cap L^\ift\cap L^1_w\right)(\bbR^3)$$ to obtain a simplified version of \eqref{eq_0213_2} as
\begin{equation}\label{eq0718_1}
    \f{d}{dt}z_c(t)= \nrm{\+{\xi}}_{L^1}^{-1}\cdot\calB\left[\+{\xi},\xi\right].
\end{equation}
 For a moment, we observe several properties of the operator $\calB$.\\

\noindent\textit{$(i)$ Bound of $\calB$:} We observe that 
$$|B[\xi_1,\xi_2]|\leq \nrm{\xi_1}_{L^2}\nrm{\calK[\xi_2]}_{L^2}$$ where $$\nrm{\calK[\xi_2]}_{L^2}^2=E[\xi_2]\lesssim \left(\nrm{r^2\xi_2}_{L^1}+\nrm{\xi_2}_{L^1\cap L^2}\right)\nrm{r^2\xi_2}_{L^1}^{1/2}\nrm{\xi_2}_{L^1}^{1/2}.$$ 
Here, we used the energy bound \eqref{eq: energy bound} in Lemma~\ref{lem: energy estimates} to obtain the last relation, and especially for the first equality, e.g., see \cite[Lemma 2.4]{Choi2024}.
Then we get the bound
\begin{equation}\label{eq_0217_1}
    |B[\xi_1,\xi_2]|\lesssim \nrm{\xi_1}_{L^2}^{\q}\left(\nrm{r^2\xi_2}_{L^1}+\nrm{\xi_2}_{L^1\cap L^2}\right)^{1/2}\nrm{r^2\xi_2}_{L^1}^{1/4}\nrm{\xi_2}_{L^1}^{1/4}.
\end{equation}

\noindent\textit{$(ii)$ Translation invariance:} By the relation 
$$\calK[\xi]\cdot\bfe_z=\f{1}{r}\rd_r\calG[\xi]$$ and the fact that $\rd_r G(r,r',z,z')$ is a function of triple of variables $(r,r',|z-z'|)$, the bilinear operator $\calB$ has the translation invariance in $z$-axis in that, by denoting the shift operator as $f^\tau:=f(\cdot-\tau\bfe_z)$ for $\tau\in\bbR$, we have
$$\calB[\xi_1,\xi_2]=\calB[\xi_1^\tau,\xi_2^\tau]\qd\mbox{for each}\q\tau\in\bbR.$$ 

\noindent\textit{$(iii)$ Traveling speed $W_H$ of Hill'
s vortex $\xi_H$:} As the solution $\Xi$ defined by 
$\Xi(t,\bfx):=\xi_H^{W_Ht}$ solves \eqref{eq: axi no swirl Euler}, we repeat the same process with \eqref{eq_0213_2} to get
\begin{equation}\label{eq0718_2}
    \begin{aligned}
    W_H=\f{d}{dt}\underbrace{\left(\f{1}{\nrm{\Xi}_{L^1}}\int_{\bbR^3}x_3\Xi\,\dd\bfx\right)}_{=W_Ht}\,=\,\nrm{\xi_H}_{L^1}^{-1}\cdot\calB\left[\Xi,\Xi\right]
    \,=\,\nrm{\xi_H}_{L^1}^{-1}\cdot\calB\left[\xi_H^{\tau(t)},\xi_H^{\tau(t)}\right].
\end{aligned}
\end{equation}

 Then, using \textit{(i)}-\textit{(iii)} above, we will obtain \eqref{eq_0214_1}. First, we combine \eqref{eq0718_1} and \eqref{eq0718_2} to observe that
$$\begin{aligned}
\left|\f{d}{dt}z_c(t)-W_H\right|\,&=\,
\left|\f{\calB\left[\+{\xi},\xi\right]}{\nrm{\+{\xi}}_{L^1}}-\f{\calB\left[\xi_H^{\tau(t)},\xi_H^{\tau(t)}\right]}{\nrm{\xi_H}_{L^1}}\right|\\
&\leq\,
\left|\f{\calB\left[\+{\xi},\xi\right]}{\nrm{\+{\xi}}_{L^1}}-\f{\calB\left[\xi_H^{\tau(t)},\xi_H^{\tau(t)}-\xi_H^{-\tau(t)}\right]}{\nrm{\xi_H}_{L^1}}\right|+\f{\calB\left[\xi_H^{\tau(t)},\xi_H^{-\tau(t)}\right]}{\nrm{\xi_H}_{L^1}}\\
&=:\, (III)\q+\q (IV).
\end{aligned}
$$ Our initial assumption \eqref{eq0718_3}, together with  \eqref{eq0307_3} and \eqref{eq0307_4}, gives 
$$\begin{aligned}
    (III)\,&\leq\, 
    \underbrace{\f{|\,\nrm{\+{\xi}}_{L^1}-\nrm{\xi_H}_{L^1}|}{\nrm{\+{\xi}}_{L^1}\nrm{\xi_H}_{L^1}}}_{\lesssim\eps}\cdot\left|\calB[\+{\xi},\xi]\right|+\underbrace{\f{1}{\nrm{\xi_H}_{L^1}}}_{\lesssim1}\cdot
    \left|\calB[\+{\xi},\xi]-\calB[\xi_H^{\tau(t)},\xi_H^{\tau(t)}-\xi_H^{-\tau(t)}]\right|\\
&\lesssim\, \eps\left|\calB[\+{\xi},\xi]\right|+
\left|\calB[\+{\xi}-
\xi_H^{\tau(t)},\,\xi]\right|+
\left|\calB[\xi_H^{\tau(t)},\,
\xi-(\xi_H^{\tau(t)}-\xi_H^{-\tau(t)})]\right|.
\end{aligned} 
 $$
 We now apply the initial assumption \eqref{eq0718_3} and the bound \eqref{eq_0217_1} to obtain that 
$$(III)\,\lesssim\,\eps.$$
 For $(IV)$, based on the observation that 
 $$|\rd_rG(r,z,r',z')|\lesssim 
 \f{r'}{[(r-r')^2+(z-z')^2]^{1/2}}+
 \f{rr'(r+r')}{[(r-r')^2+(z-z')^2]^{3/2}}
 \lesssim_{r,r'}\f{1}{|z-z'|}+\f{1}{|z-z'|^3}$$ which follows directly from the definition \eqref{eq: Green function formula}, we obtain that 
 $$\begin{aligned}
     |(IV)|\,&=\, \left|\f{1}{\nrm{\xi_H}_{L^1}}\int_{\bbR^3}\f{1}{r}\xi_H^{\tau(t)}\rd_r\calG[\xi_H^{-\tau(t)}]\,\dd\bfx\right|\\
     &\lesssim\,\f{1}{\nrm{\xi_H}_{L^1}}\int_{\Pi\times\Pi}
     \left|\rd_rG(r,z,r',z')\right|
     \xi_H^{\tau(t)}(r,z)\xi_H^{-\tau(t)}(r',z') r'\dd r'\dd z'\dd r\dd z\\
     &\lesssim\, \int_{\Pi\times\Pi}\left[\f{1}{|z-z'|}+\f{1}{|z-z'|^3}\right] \xi_H^{\tau(t)}(r,z)\xi_H^{-\tau(t)}(r',z')\dd r'\dd z'\dd r\dd z\\
     &=\, \int_{\Pi\times\Pi}\left[\f{1}{|z-z'+2\tau(t)|}+\f{1}{|z-z'+2\tau(t)|^3}\right] \xi_H(r,z)\xi_H(r',z')\dd r'\dd z'\dd r\dd z\\
     &\lesssim\, \f{1}{|\tau(t)-1|}+\f{1}{|\tau(t)-1|^3}
 \end{aligned}$$
  which leads to 
 $$\left|(IV)\right|\lesssim \f{1}{\tau(t)-1}$$ given that $\tau(t)\geq2.$ It proves \eqref{eq_0214_1}:
 $$\left|\f{d}{dt}z_c(t)-W_H\right|\,\leq\,
 (III)+(IV)\,\lesssim\,\eps+\f{1}{\tau(t)-1}\qd\mbox{given that}\q \tau(t)\geq2.$$
 \\

 \noindent\textbf{[Step 3] The estimate \eqref{eq: z(t) vs Wt} holds:}
$$|z_c(t)-z_c(0)-W_Ht|< C_3\eps t\qd\mbox{for each }t\geq0$$
\textbf{for some uniform constant $C_3>0$.}

We observe that \eqref{eq: z(t) vs tau(t)} implies 
\begin{equation}\label{eq_0217_3}
    \tau(t)\geq z_c(t)-C_0\eps(t+1)\qd\mbox{for }t\geq0
\end{equation} for some constant $C_0=C_0(\calM)>0$. Since \eqref{eq0307_1} and \eqref{eq0306_2} give
$$\begin{aligned}
    z_c(t)\,&\geq\,
z_c(0)-\sup_{s\geq0}\nrm{\bfu(s)}_{L^\ift}t\\
&>\, d-\eps-C_1t\cdot\nrm{\xi_0}_{L^\ift(\bbR^3)}^{1/2}
        \nrm{\xi_0}_{L^1(\bbR^3)}^{1/4}
        \nrm{r^2\xi_0}_{L^1(\bbR^3)}^{1/4}\\
        &\geq\, d-\eps-C_4t \calM^{1/2}t
\end{aligned} 
$$ for some uniform constant $C_4>0$, the inequality \eqref{eq_0217_3} leads to 
$$\tau(t)\geq d-t\cdot\left(C_4\calM^{1/2}+\eps C_0\right)-\eps(1+C_0).$$ We suppose that $\eps>0$ is small enough to satisfy $$\eps(1+C_0)\leq \f{1}{4}$$ and apply $ d>\eps^{-1}>1$ to get 
$$\tau(t)\geq \f{3}{4}d-t\cdot\left(C_4M^{1/2}+\eps C_0\right).$$
So there exists $T>0$ such that 
$$\tau(t)>\f{d}{2}\qd\mbox{for each}\q t\in[0,T].$$ For each $t\in[0,T]$, by assuming $\eps<1/4$, we get
$$\tau(t)>\f{d}{2}\geq\f{d_0(\eps)}{2}>\f{1}{2\eps}>2,$$ and we can use \eqref{eq_0214_1} to obtain
\begin{equation}\label{eq_0219_2}
    \left|\f{d}{dt}z_c(t)-W_H\right|<C_2\left(\eps+\f{4}{d}\right)\qd\mbox{for each}\q t\in[0,T]
\end{equation} and so 
\begin{equation}\label{eq_0217_5}
    |z_c(t)-z_c(0)-W_Ht|< C_2\left(\eps+\f{4}{d}
    \right)\cdot t\qd\mbox{for each}\q t\in[0,T].
\end{equation} 
It remains to show that the above inequality holds for all $t\geq0$, since the relation $d\geq d_0(\eps)>\eps^{-1}$ will be combined with the above relation to obtain \eqref{eq: z(t) vs Wt}. We assume the contrary, i.e., we suppose that \eqref{eq_0217_5}
 ceases to hold for the first time at $t=\tld{T}\geq T$ for some $\tld{T}<\ift$, i.e.,

$$\tld{T}:= \sup\left\{ T\geq0:\, 
|z_c(t)-z_c(0)-W_Ht|< C_2\left(\eps+\f{4}{d}
    \right)\cdot t
\qd\mbox{for each}\q t\in[0,T]\right\}<\ift.$$ Then the continuity of $z_c(\cdot)$ gives 
\begin{equation}\label{eq_0219_1}
    |z_c(t)-z_c(0)-W_Ht|< C_2\left(\eps+\f{4}{d}
    \right)\cdot t 
\qd\mbox{for each}\q t\in[0,\tld{T})
\end{equation}
and 
\begin{equation}\label{eq_0217_6}
    |z(\tld{T})-z_c(0)-W_H\tld{T}|= 
    C_2\left(\eps+\f{4}{d}
    \right)\cdot\tld{T}.
\end{equation}
Then \eqref{eq_0219_1} and \eqref{eq0307_1} lead to 
$$\begin{aligned}
    z_c(t)\,&>\, z_c(0)+\left(W_H-C_2\left(\eps+4/d\right)\right)t \\
    &\geq\, d-\eps+\left(W_H-5C_2\eps\right)t
    \qd\mbox{for each}\q t\in[0,\tld{T})
\end{aligned}
$$ due to $d\geq d_0(\eps)>\eps^{-1}$, and combining this result with \eqref{eq_0217_3} gives us 
$$\begin{aligned}
    \tau(t)\,&\geq\, 
    z_c(t)-C_0\eps(t+1)\\
    &\geq\, d-\eps+\left(W_H-5C_2\eps\right)t -C_0\eps(t+1) \\
    &=\, d-\underbrace{\eps(1+C_0)}_{(V)}+\underbrace{\left(W_H-(5C_2+C_0)\eps\right)}_{(VI)}t
    \qd\mbox{for each}\q t\in[0,\tld{T}).
\end{aligned}$$ 
Recall that we have assumed $\eps>0$ is small enough such that $(V)< \f{1}{4}d.$ So we get 
$$\tau(t)> \f{3}{4}d+(VI)\cdot t
\qd\mbox{for each}\q t\in[0,\tld{T}).$$
We further assume $\eps>0$ is small enough to get $(VI)>0$ (the smallness of $\eps>0$ depends only on $C_2>0$ and $C_0=C_0(\calM)>0$, and so it depends only on $\calM>0$). Then we get 
$$\tau(t)> \f{3}{4}d>2
\qd\mbox{and so}\qd
\tau(t)-1>\f{1}{2}d
\qd\mbox{for each}\q t\in[0,\tld{T})$$ ($\eps>0$ is already assumed to be small enough to guarantee $d\geq d_0(\eps)>\eps^{-1}>4$).
Applying it to \eqref{eq_0214_1} gives 
$$|z_c(\tld{T})-z_c(0)-W_H\tld{T}|< C_2\left(\eps+\f{2}{d}
    \right)\cdot \tld{T}$$
    after integration, which contradicts \eqref{eq_0217_6}. Hence, we get $\tld{T}=\ift,$ which means that the inequality 
    \eqref{eq_0217_5} holds for all $T>0$, i.e., we obtain \eqref{eq: z(t) vs Wt}:
    $$|z_c(t)-z_c(0)-W_Ht|< C_2\left(\eps+\f{4}{d}
    \right)\cdot t<
    5C_2\eps t\qd\mbox{for each}\q t\geq0$$
    due to $d>\eps^{-1}$. It completes the proof.
\end{proof}

\q \section*{Acknowledgments}

We deeply appreciate Prof. Kyudong Choi (UNIST) for introducing this problem.

\end{document}